\documentclass[11pt]{amsart}

\usepackage{amsmath,amsfonts}
\usepackage{amssymb}
\usepackage{amscd}
\usepackage{amsthm}
\usepackage{yhmath}
\usepackage{subfigure}
\usepackage{graphicx}   
\usepackage{comment}
\usepackage{stackrel}

\usepackage[normalem]{ulem}
\usepackage{soul}

\usepackage{appendix}

\usepackage[all]{xy}
\usepackage{color}
\usepackage[dvipsnames]{xcolor}

\usepackage{marginnote}
\usepackage{hyperref}

\usepackage{todonotes}

\numberwithin{equation}{section}

\newtheorem{theorem}{Theorem}[section]

\newtheorem{proposition}[theorem]{Proposition}

\newtheorem{definition}[theorem]{Definition}
\newtheorem{lemma}[theorem]{Lemma}
\newtheorem{corollary}[theorem]{Corollary}

\theoremstyle{remark}
\newtheorem{remark}[theorem]{Remark}

\newcommand{\p}{\partial}
\renewcommand{\varphi}{\tau}

\def\XXint#1#2#3{{\setbox0=\hbox{$#1{#2#3}{\int}$}
	\vcenter{\hbox{$#2#3$}}\kern-.5\wd0}}

\newcommand{\free}{\mathfrak{f}_{2,3}}
\newcommand{\Ker}{\operatorname{Ker}}
\newcommand{\Lone}{\Lambda^1 \R^3}
\newcommand{\Ltwo}{\Lambda^2 \R^3}

\DeclareMathOperator{\Hor}{Hor}
\newcommand{\step}[1]{\par\medskip\noindent\it#1\rm}
\renewcommand{\O}{\mathcal{O}}

\newcommand{\R}{\mathbb R}

\newcommand{\g}{\mathfrak{g}}

\newcommand{\Int}{\operatorname{Int}}
\renewcommand{\ker}{\operatorname{Ker}}

\newcommand{\Lie}{\operatorname{Lie}}

\newcommand{\rank}{\operatorname{rank}}

\DeclareMathOperator{\Char}{Char}
\DeclareMathOperator{\Nonchar}{Nonchar}

\newcommand{\e}{\varepsilon}

\newcommand{\Anh}{\operatorname{Anh}}

\newcommand{\wh}{\widehat}

\newcommand{\Span}{\operatorname{span}}

\renewcommand{\epsilon}{\varepsilon}

\renewcommand{\tilde}{\widetilde}

\newcommand{\ol}{\overline}

\newcommand{\calO}{\mathcal{O}}
\newcommand{\calU}{\mathcal{U}}

\newcommand{\Per}{\operatorname{Per}}

\def\om{\omega}

\newcommand{\frakf}{\mathfrak{f}}
\newcommand{\frakg}{\mathfrak{g}}
\newcommand{\frakh}{\mathfrak{h}}

\begin{document} 

\title[Precisely monotone sets in step-2 rank-3 Carnot algebras]{Precisely monotone sets in step-2 rank-3\\ Carnot algebras}

\author{Daniele Morbidelli}
\address[Morbidelli]{Dipartimento di Matematica, Alma Mater Studiorum Universit\`a di Bologna, Italy}
\email{daniele.morbidelli@unibo.it}

\author{S\'everine Rigot}
\address[Rigot]{Universit\'e C\^ote d'Azur, CNRS, LJAD, France}
\email{Severine.RIGOT@univ-cotedazur.fr}

%

\begin{abstract}
 A subset of a Carnot group is said to be precisely monotone if the restriction of its characteristic function to each integral curve of every left-invariant horizontal vector field is monotone. Equivalently, a precisely monotone set is a h-convex set with h-convex complement. Such sets have been introduced and classified in the Heisenberg setting by Cheeger and Kleiner in the 2010's. In the present paper, we study precisely monotone sets in the wider setting of step-2 Carnot groups, equivalently step-2 Carnot algebras. In addition to general properties, we prove a classification in terms of sublevel sets of h-affine functions in step-2 rank-3 Carnot algebras that can be seen as a generalization of the one obtained by Cheeger and Kleiner in the Heisenberg setting. There is however a significant difference here as it is known that, unlike the Heisenberg setting, there are sublevel sets of h-affine functions on the free step-2 rank-3 Carnot algebra that are not half-spaces.
\end{abstract}

\maketitle


\section{Introduction} \label{sect:introduction}

Monotone sets have been first introduced by Cheeger and  Kleiner in \cite{CheegerKleiner10} where the proof of the non biLipschitz embeddability of the first Heisenberg group into $L^1$ is reduced to the classification of its monotone subsets, see also~\cite{CheegerKleinerNaor11}. Later on, this classification together with related notions of monotonicity/non-monotonicity appeared in a crucial way in several works related to geometric measure theory issues in the Heisenberg setting, see for instance~\cite{NaorYoung}, \cite{FasslerOrponenRigot}, \cite{naor2021foliated}, \cite{young2021areaminimizing}.

In the perspective of a further analysis along these lines of research in more general settings, we study here precisely monotone sets in more general Carnot groups, see for instance~\cite[Section~2.1]{MR3587666} and the references therein for an introduction to Carnot groups. Besides their relevance in the aforementioned questions, let us stress that monotone sets have also their own interest. They can for instance be proved to be local minimizers for the intrinsic perimeter, see \cite[Proposition~3.9]{young2021areaminimizing} and Proposition~\ref{prop:minimizers}.   Let us also mention that sets with constant horizontal normal, widely studied in connection with the theory of sets with locally finite intrinsic perimeter,  are examples of monotone sets, see the pioneering works \cite{FSSC01}, \cite{FSSC03}, and \cite{MR2496564}, \cite{MR3330910}, \cite{bellettini2019sets}.

A subset $E$ of a Carnot group is said to be precisely monotone if the restriction of its characteristic function to each integral curve of every left-invariant horizontal vector field is monotone when seen as a function from $\R$ to $\R$. In other words, the image of any such curve intersects both $E$ and its complement $E^c$ in a connected set, equivalently, both $E$ and $E^c$ are h-convex, see for instance~\cite{Rickly06},~\cite{ArenaCarusoMonti12},~\cite{CalogeroPini19},~\cite{MR4119259}, for more details about h-convex sets. Despite the simplicity of their definition, precisely monotone sets turn out to be rather difficult to describe. 

On the one hand, a classification of precisely monotone sets is known so far only in some particular settings. Namely, it has been proved in~\cite{CheegerKleiner10} that if $E$ is a non empty precisely monotone strict subset of the first Heisenberg group $\mathbb{H}$ then there is an open half-space $C$ such that $C\subset E \subset \overline{C}$. This classification has been generalized to higher dimensional Heisenberg groups in~\cite{NaorYoung} and to Carnot groups of M\'etiver's type and the direct product $\mathbb{H} \times \R$ in \cite{MR4119259}. These are, at least to our knowledge, the only cases where a classification for precisely monotone sets has been established. On the other hand, there are plenty of examples of Carnot groups, such as the free one of step-2 and rank-3, where such a classification in terms of half-spaces is known to be false, as we will explain below.

Going back to arbitrary Carnot groups, we say that a real-valued function is  horizontally monotone, h-monotone in short, if it is monotone along all integral curves of left-invariant horizontal vector fields when seen as a function from $\R$ to $\R$. It follows from the very definitions that sublevel sets of h-monotone functions are precisely monotone.  It is then natural to ask whether a classification of precisely monotone sets can be given in terms of sublevel sets of h-monotone functions. Note that the one obtained in \cite{CheegerKleiner10,NaorYoung,MR4119259} fits such a classification as an open half-space can always be written as a sublevel set of some affine function and since affine functions on step-2 Carnot groups are h-affine and hence h-monotone.

In the present paper we consider step-2 Carnot groups, identified with step-2 Carnot algebras, see Section~\ref{sect:PM-step2-Carnot-algebras} for our convention about the natural identification between step-2 Carnot groups and algebras. We first prove general properties of precisely monotone subsets of arbiratry step-2 Carnot algebras. They strongly rely on Cheeger-Kleiner's classification in the Heisenberg case together with the fact that integral curves of left-invariant horizontal vector fields in step-2 Carnot algebras are 1-dimensional affine subspaces, called horizontal lines. We next classify measurable precisely monotone subsets of step-2 rank-3 Carnot algebras in terms of sublevel sets of h-affine functions, see Theorems~\ref{thm:main} and~\ref{thm:nonfree-case}. We recall that if $\frakg$ is a step-2 Carnot algebra, a function $\phi:\frakg \rightarrow \R$ is said to be horizontally affine, h-affine in short, if its restriction to every horizontal line is affine (see \cite{LeDonneMorbidelliRigot1}). Obviously h-affine functions are h-monotone. Therefore Theorems~\ref{thm:main} and~\ref{thm:nonfree-case} give a positive answer to the question of the classification of measurable precisely monotone sets in terms of sublevel sets of h-monotone functions in the step-2 rank-3 cases that actually involves a a priori smaller class of functions. We stress that the free step-2 rank-3 case is an example of a Carnot algebra where there are h-affine functions that are not affine (see \eqref{e:phi-introduction} and \cite{LeDonneMorbidelliRigot1} for a complete description of such examples) whose sublevel sets are not half-spaces and where a classification of precisely monotone sets in terms of half-spaces can therefore not hold. This creates  in particular significant differences compared to the settings considered in \cite{CheegerKleiner10,NaorYoung,MR4119259}.

In the free step-2 rank-3 Carnot algebra $\free = \Lone \oplus \Ltwo$ equipped with the Lie bracket for which the only non trivial relations are given by $[\theta,\tau]:=\theta \wedge \tau$ for $\theta,\tau\in \Lone$ and the induced group law given by $(\theta+\omega)\cdot (\tau+\zeta):= \theta+\tau+\omega+\zeta+[\theta,\tau]$ for $\theta,\tau\in\Lone$, $\omega,\zeta\in\Ltwo$ (see Section~\ref{sect:PM-step2-Carnot-algebras}), the classification reads as follows.

\begin{theorem} \label{thm:main} Let $E\subset \free$ be precisely monotone and measurable. Then either $E=\emptyset$, $E=\free$, or there is a non constant h-affine function $\phi:\free \rightarrow \R$ such that 
\begin{equation} \label{e:main}
\Int(E) = \{x\in \free :\, \phi(x) < 0\} \quad \text{and} \quad \overline{E}= \{x\in \free :\, \phi(x) \leq 0\}~.
\end{equation}
\end{theorem}

More explicitly, we prove that given a non empty measurable precisely monotone strict subset $E$ of $\free $ and given $\nu \in \Lambda^3\R^3 \setminus \{0\}$ there is $(\eta_0, \eta_1,\eta_2,\eta_3)\in \Lambda^0\R^3\times\Lambda^1\R^3\times\Lambda^2\R^3 \times \Lambda^3 \R^3$ with $(\eta_0, \eta_1,\eta_2)\not=(0,0,0)$ such that \eqref{e:main} holds true with $\phi$ given by
\begin{equation}  \label{e:phi-introduction}
\phi(\theta+\omega) \nu = \eta_3 + \eta_2 \wedge \theta + \eta_1 \wedge \omega + \eta_0 \theta \wedge \omega
\end{equation}
for $\theta\in\Lone$, $\omega\in \Ltwo$. Such a quadratic function can easily be seen to be h-affine. In addition, let us mention that we will also get from our arguments that $\Int(E) = \{x\in \free :\, \phi(x) < 0\}$ and $\Int(E^c) = \{x\in \free :\, \phi(x) > 0\}$ are the two connected components of $(\partial E)^c = \{x\in \free :\, \phi(x) \not= 0\}$.

\smallskip

Next, writing a step-2 rank-3 Carnot algebra as a quotient of $\free$ and using the fact that h-affine functions on a proper quotient of $\free$ are affine (see~\cite{LeDonneMorbidelliRigot1}) we shall deduce from Theorem~\ref{thm:main} the following classification in nonfree step-2 rank-3 Carnot algebras.

\begin{theorem} \label{thm:nonfree-case}
Let $\frakg$ be a step-2 rank-3 Carnot algebra and assume that $\frakg$ is not isomorphic to $\free$. Let $E\subset \frakg$ be precisely monotone and measurable. Then either $E=\emptyset$, $E=\frakg$, or there is an open half-space $C$ such that $C\subset E \subset \overline{C}$.
\end{theorem}

Note that a step-2 rank-3 Carnot algebra $\frakg$ that is not isomorphic to $\free$ is either isomorphic to $\mathbb{H} \times \R$ or to $\free / \mathfrak{i}$ where $\mathfrak{i}$ is an ideal in $\free$ generated by an element in $\Ltwo \setminus\{0\}$. If $\frakg$ is isomorphic to $\mathbb{H} \times \R$, we recover the classification proved in \cite{MR4119259}. If $\frakg$ is isomorphic to $\free / \mathfrak{i}$ with $\mathfrak{i}$ an ideal in $\free$ generated by an element in $\Ltwo \setminus\{0\}$ then we need to make use of Theorem \ref{thm:main} to get the classification given by Theorem \ref{thm:nonfree-case}, as it can indeed not be deduced from the previously known cases studied in \cite{CheegerKleiner10,NaorYoung,MR4119259}.

Before we give a sketch of the proof of Theorem~\ref{thm:main} and discuss possible generalizations to step-2 Carnot algebras of higher rank, let us say a few words about the step-3 or higher setting. It should be noticed that integral curves of left-invariant horizontal vector fields in step-3 or higher Carnot algebras are not necessarily 1-dimensional affine subspaces. Among other things,  this is expected to create significant differences compared to the step-2 setting. Examples given in \cite{bellettini2019sets}, see the discussion in \cite{antonelli2020polynomial}, suggest that there may be step-3 Carnot algebras where one cannot classify precisely monotone sets in terms of sublevel sets of h-affine functions (note however that in \cite{antonelli2020polynomial} only locally integrable h-affine functions are considered). To our knowledge, the question of a classification in terms of sublevel sets of h-monotone functions in step-3 or higher remains however open, and we shall not pursue in this direction here.
 
Going back to the free step-2 rank-3 framework, let us now explain, without entering the technical details, the main ideas behind the proof of Theorem~\ref{thm:main}. It will be articulated into two main steps. First,  we will prove that \eqref{e:main} holds true locally near noncharacterictic points of the boundary, see~\eqref{e:characteristic-points} for the definition of noncharacteristic points and Proposition~\ref{prop:local} for a precise statement. The argument is based on a local representation proved in \cite{MR4119259} of the boundary $\partial E$ of a precisely monotone set $E$ as an intrinsic graph in the sense of~\cite{MR3511465} near non characteristic points. Making use of Cheeger-Kleiner's classification in suitable Heisenberg subalgebras of $\free$, we show through a careful analysis that this local representation can be written as a level set of some h-affine function. In a second step, we use monotonicity more globally to upgrade the local representation into a global one, showing that $\partial E$ is the zero level set of some function $\phi:\free \rightarrow \R$ of the form \eqref{e:phi-introduction}. Both inclusions $ \phi^{-1}(0) \subset \partial E$ and $\p E\subset \phi^{-1}(0)$ are nontrivial and require a careful analysis. 

Concerning a possible generalization of our strategy to free step-2 Carnot algebras of higher rank, although some of our arguments extend to this more general framework, it is however not entirely clear to us whether the whole strategy does. To give an idea of some of the issues in higher rank, let us mention that in the free step-2 rank-$n$ Carnot algebra the horizontal space at some given point is a $n$-dimensional affine subspace, whereas the dimension of the whole space is $n(n+1)/2$ and hence increases quadratically with respect to $n$. As a consequence, lying on some horizontal line for a pair of points (this obviously plays a key role for our purposes) becomes a more and more rare circumstance as the rank increases. We however plan to devote future works to step-2 higher rank cases. We also refer to Remark \ref{rmk:monotone-sets} for the relationship between precisely monotone and monotone sets.

The rest of this paper is organized as follows. In Section~\ref{sect:PM-step2-Carnot-algebras} we prove several properties of precisely monotone subsets of step-2 Carnot algebras. In Sections~\ref{sect:local-statement-free-case} to~\ref{sect:(sub)levelsets-h-affine-maps} we focus on the free step-2 rank-3 case. As already explained we first prove in Section~\ref{sect:local-statement-free-case} a local description of the boundary of a precisely monotone subset of $\free$ near noncharacteristic points. In Section~\ref{sect:global-statement-free-case} we upgrade this local statement into a global one and  we  conclude the proof of Theorem~\ref{thm:main}. Properties of level and sublevel sets of h-affine functions on $\free$ that may have their own interest and play a major role in Sections~\ref{sect:local-statement-free-case} and~\ref{sect:global-statement-free-case} are proved in Section~\ref{sect:(sub)levelsets-h-affine-maps}. The final Section~\ref{sect:nonfree-case} is devoted to the proof of Theorem~\ref{thm:nonfree-case} that will be obtained as a rather easy consequence of Theorem~\ref{thm:main}.

%
\textit{Acknowledgements.} The authors are grateful to E. Le Donne for several useful discussions.
\section{Precisely monotone sets in step-2 Carnot algebras} \label{sect:PM-step2-Carnot-algebras}
    
In this section we establish several properties of precisely monotone subsets of step-2 Carnot algebras. Most of these properties will be used in the next sections to study precisely monotone subsets of the free step-2 rank-3 Carnot algebra. 

We recall that a Lie algebra $\frakg$ -- always assumed to be real and finite dimensional in this paper -- is said to be nilpotent of step 2 if the derived algebra $\frakg_2 := [\frakg,\frakg]$ is non trivial, i.e., $\frakg_2\not=\{0\}$, and central, i.e., $[\frakg,\frakg_2] = \{0\}$. Here, given $U, V \subset \frakg$, we denote by $[U,V]$ the linear subspace of $\frakg$ generated by elements of the form $[u,v]$ with $u\in U$ and $v\in V$. A step-2 Carnot algebra $\frakg$ is a Lie algebra nilpotent of step 2 that is equipped with a stratification, namely, $\frakg = \frakg_1 \oplus \frakg_2$ where $\frakg_1$ is a linear subspace of $\frakg$ that is in direct sum with $\frakg_2$. Note that $[\frakg_1,\frakg_1] = \frakg_2$. The rank of $\frakg$ is defined as $\rank \frakg:= \dim \frakg_1$.   Such a Lie algebra is naturally endowed with the group law given by $x\cdot y := x+y+[x,y]$ for $x,y \in \frakg$ that makes it a step-2 Carnot group. It is actually well known that any step-2 Carnot group can be realized in this way. We shall therefore view a step-2 Carnot algebra both as a Lie algebra and group.

We fix from now on in this section a step-2 Carnot algebra $\g = \frakg_1 \oplus \frakg_2$. Given a scalar $t\in \R$ and an element $x\in \g$, we set $x^t := tx$. We say that a set $\ell\subset \frakg$ is a horizontal line if there are $x\in \frakg$ and $y\in\frakg_1 \setminus \{0\}$ such that $\ell =\{x\cdot y^t\in \frakg:\, t\in \R\}$. Note that since $x\cdot y^t = x + t(y+[x,y])$, horizontal lines are 1-dimensional affine subspaces of $\frakg$.

\begin{definition} \label{def:PM}
A set $E\subset \frakg$ is said to be precisely monotone if every horizontal line intersects both $E$ and $E^c$ in a connected set.
\end{definition}

\begin{remark} \label{rmk:monotone-sets} Monotone sets are defined in the same way as precisely monotone sets except that the condition given in Definition~\ref{def:PM} is required to hold true only for almost every horizontal line $\ell$, and up to a null set within $\ell$. For simplicity, we restrict ourselves in the present paper to precisely monotone sets, whose study should be sufficient to give the key ideas towards a classification of monotone sets (see for instance~\cite[Sect.4-5]{CheegerKleiner10}).
\end{remark}

Note that $E$ is precisely monotone if and only if $E^c$ is precisely monotone. Note also that if $E$ is precisely monotone and $x\in \frakg$ then $x\cdot E$ is precisely monotone. We first recall some known facts.

\begin{proposition}[{\cite[Proposition~4.6]{CheegerKleiner10}}{\cite[Proposition~3.3]{MR4119259}}] \label{tretre} 
  Let $E\subset\frakg$ be precisely monotone. If $x\in\p E$ and $y\in \frakg_1 \setminus \{0\}$ are such that $x\cdot y\in\Int( E)$ then $\{x\cdot y^t \in \frakg : t>0\} \subset \Int( E)$ and  
  $\{x\cdot y^t \in \frakg : t<0\}\subset \Int( E^c)$. The same statement holds true with the role of $E$ and $E^c$ exchanged.
\end{proposition}

\begin{lemma} [{\cite[Lemma~4.8]{CheegerKleiner10}}{\cite[Lemma~3.4]{MR4119259}}] \label{linee} 
 Let $E\subset\frakg$ be precisely monotone. If $\ell$ is a horizontal line such that $\ell \cap \partial E$ contains more than one point then $\ell \subset \partial E$.
 \end{lemma}
 
We say that a Lie algebra is a Heisenberg algebra if it is a step-2 rank-2 Carnot algebra. 

\begin{theorem}[{\cite[Theorem~4.3]{CheegerKleiner10}}] \label{thm:CK}
Let $\frakh$ be a Heisenberg algebra and $E\subset \frakh$ be precisely monotone. Then either $E=\emptyset$, $E=\frakh$, or there is an open half-space $C$ such that $C\subset E \subset \overline{C}$. In particular $\partial E$ is either empty or a 2-dimensional affine subspace of $\frakh$.
\end{theorem}
 
We define the horizontal space at a point $x\in \frakg$ as $\Hor_x:=x\cdot \frakg_1$. In other words, $\Hor_x$ is the union of all horizontal lines in $\frakg$ containing $x$. Note that $\Hor_x$ can also easily be seen to be an affine subspace of $\frakg$ of dimension  equal to $\rank \frakg$. In the next lemma we prove that the precise monotonicity of a set $E$ induces a structure of affine subspace on $\Hor_{x} \cap \partial E$ when $x\in \partial E$. In the lemma below and in  the rest of this paper, given $A\subset B \subset  \frakg$, we denote by $\Int_B (A)$ and $\partial_B A$ the relative interior and boundary in $B$ of a subset $A$ of $B$ with respect to the induced topology.

\begin{lemma}\label{apro} 
 Let $E\subset\frakg$ be precisely monotone and $x\in \partial E$. Then $\Hor_{x} \cap \partial E$ is an affine subspace of $\,\Hor_x$. 
\end{lemma} 

\begin{proof}
Taking Lemma~\ref{linee} into account, we need to prove that if $\ell_1$ and $\ell_2$ are horizontal lines, $\ell_1 \not=\ell_2$, such that $\ell_1\cup\ell_2\subset\p E$ and $x\in \ell_1\cap \ell_2$ then the affine subspace generated by $\ell_1 \cup \ell_2$ is contained in $\partial E$. Using a left-translation, we can assume with no loss of generality that $x=0$, $\ell_i = \{y_i^t \in \frakg : \, t\in \R\}$ for some linearly independent $y_i\in \frakg_1 \setminus \{0\}$ for $i=1,2$, and we shall prove that
\begin{equation} \label{e:key-lemma}
\ell_1 \oplus \ell_2 \subset \partial E~.
\end{equation}

If $[y_1,y_2] = 0$ then every point in $\ell_1 \oplus \ell_2 = \ell_1 \cdot \ell_2$  lies on a horizontal line that intersects both $\ell_1 \setminus \{0\}$ and $\ell_2 \setminus \{0\}$. Namely, 
\begin{equation*}
\ell_1 \cdot \ell_2 = \bigcup_{t\not=0} (\ell_{t} \cup \tilde{\ell}_t)
\end{equation*}
where $\ell_{t}:= y_1^t \cdot \{(y_1\cdot y_2)^s \in \frakg :\ s\in \R\}$ and $\tilde{\ell}_{t}:= y_1^t \cdot \{(y_1^{-1}\cdot y_2)^s \in \frakg :\ s\in \R\}$. For every $t\not=0$ we have $y_1^t \in \ell_t \cap \tilde{\ell}_t \cap (\ell_1 \setminus \{0\})$, $y_2^{-t} \in \ell_t \cap (\ell_2 \setminus \{0\})$, and $y_2^t \in \tilde{\ell}_t \cap (\ell_2 \setminus \{0\})$.  Therefore, both $\ell_t$ and $\tilde{\ell}_t$ intersect $\p E$ in at least two points and~\eqref{e:key-lemma}  follows from Lemma~\ref{linee}.

\smallskip

If $[y_1,y_2] \not= 0$, we denote by $\frakh:=\Span\{y_1,y_2\} \oplus \Span\{[y_1,y_2]\}$ the Lie subalgebra of $\frakg$ generated by $y_1$ and $y_2$ and we consider the family of horizontal lines
\begin{equation*}
  \ell^b:= y_2^b \cdot \ell_1 \subset \frakh~.
\end{equation*}
where $b\in \R$. We distinguish two cases.

\step{Case~1.} If there is $b\neq 0$ such that $ \ell^b \subset  \p E$ then the horizontal lines $\ell^{0} = \ell_1$ and $ \ell^b$ are parallel with distinct projection in the Heisenberg algebra $\frakh$ in the sense of~\cite{CheegerKleiner10} and both contained in~$\p E$. Then Lemma~\ref{linee} together with~\cite[Lemma~4.10]{CheegerKleiner10} applied to the set $\mathcal{G}:= \frakh \cap \p E$ implies that $\frakh \subset \p E$. Then~\eqref{e:key-lemma} follows since $\ell_1 \oplus \ell_2 \subset \frakh$.

\smallskip

\step{Case~2.} If \textit{Case~1} does not hold, since $y_2^b \in \ell_2 \cap  \ell^b \subset \partial E$, we get  from Proposition~\ref{tretre} that for every $b\not=0$ either
\begin{equation} \label{e:key-lemma-case2-1}
 \ell^b_+  \subset \Int(E) \cap \frakh \subset \Int_\frakh (E\cap \frakh) \quad \text{and} \quad  \ell^b_- \subset \Int(E^c) \cap \frakh \subset \Int_\frakh (E^c\cap \frakh) 
\end{equation} 
or
\begin{equation} \label{e:key-lemma-case2-2}
 \ell^b_-  \subset \Int(E) \cap \frakh \subset \Int_\frakh (E\cap \frakh) \quad \text{and} \quad  \ell^b_+ \subset \Int(E^c) \cap \frakh \subset \Int_\frakh (E^c\cap \frakh) 
\end{equation} 
where $ \ell^b_+:=y_2^b \cdot \{y_1^t\in \frakg : \, t>0 \}$ and $ \ell^b_-:=y_2^b \cdot \{y_1^t \in \frakg : \, t<0 \}$. It follows that for every $b\not=0$
\begin{equation*}
\{y_2^b \} = \overline{ \ell^b_+} \cap \overline{ \ell^b_-} \subset \partial_\frakh (E\cap \frakh)~.
\end{equation*} 
Since $\partial_\frakh (E\cap \frakh)$ is a closed subset of $\frakh$, we get that $\ell_2 \subset \partial_\frakh (E\cap \frakh)$. Since $E\cap \frakh$ is a monotone subset of the Heisenberg algebra $\frakh$, it follows from Theorem~\ref{thm:CK} that $\partial_\frakh (E\cap \frakh)$ is a 2-dimensional linear subspace of $\frakh$ that contains $\ell_2$, i.e., there is $(p,q) \in \R^2 \setminus \{(0,0)\}$ such that
\begin{equation} \label{e:key-lemma-boundary}
 \partial_\frakh (E\cap \frakh)=\{sy_1+t y_2 + u [y_1, y_2] \in \frakh: \,s,t,u\in\R,\, ps+qu=0\}~.
\end{equation}
We now verify that $pq=0$. We argue by contradiction and assume that $p\not=0$ and $q\not=0$. Then~\eqref{e:key-lemma-boundary} implies that $\ell^{p/q}  = \{sy_1+q^{-1}p y_2 - sq^{-1}p [y_1, y_2] \in \frakh: \, s \in \R\} \subset \partial_\frakh (E\cap \frakh)$ which contradicts both~\eqref{e:key-lemma-case2-1} and~\eqref{e:key-lemma-case2-2} for $b=p/q$. Therefore $pq=0$. If $p=0$, we get from~\eqref{e:key-lemma-boundary} that $\partial_\frakh (E\cap \frakh) = \ell_1 \oplus \ell_2$ and~\eqref{e:key-lemma} follows since $\partial_\frakh (E\cap \frakh) \subset \partial E$. If $q=0$, we get from~\eqref{e:key-lemma-boundary} that $\ell_2 \oplus \R[y_1,y_2] = \partial_\frakh (E\cap \frakh) \subset \frakh \cap \partial E$. In particular the horizontal line $\ell:= [y_1,y_2] \cdot \{y_2^t \in \frakg : \, t\in \R\}$ is contained in $\frakh \cap \partial E$. It follows that $\ell$ and $\ell_1$ are skew lines in the sense of~\cite{CheegerKleiner10} that are contained in $\frakh \cap \partial E$. Then Lemma~\ref{linee} together with~\cite[Lemma~4.10]{CheegerKleiner10} applied to the set $\mathcal{G}:= \frakh \cap \p E$ implies that $\frakh \subset \p E$. Therefore~\eqref{e:key-lemma} follows since $\ell_1 \oplus \ell_2 \subset \frakh$.
\end{proof}

Given $S\subset \frakg$ we set
\begin{equation}\label{e:characteristic-points}
\Char(S):=\{ x\in S: \,\Hor_x\cap S=\Hor_x  \} \,\,\, \text{and} \,\,\, \Nonchar(S):= S \setminus \Char(S)~. 
\end{equation}
Note that if $S$ is closed then $\Char(S)$ is closed and in such a case $\Nonchar(S)$ is therefore a relatively open subset of  $S$.

In the next proposition, we upgrade Lemma~\ref{apro} proving that for $x\in \Nonchar(\p E)$ we have $\dim (\Hor_x  \cap \p E) = \dim \Hor_x  -1$.

\begin{proposition}  \label{foglietto} 
Let $E\subset\frakg$ be precisely monotone. Then 
\begin{equation*} 
\Nonchar(\p E)=\{ x\in\p E: \, \Hor_x\cap \p E  \text{ is a }  \text{codimension-1 affine subspace of } \Hor_x\}~.
\end{equation*}
\end{proposition}
 
\begin{proof}
We know from Lemma~\ref{apro} that for every $x\in \partial E$ the set $\Hor_x\cap \p E $ is an affine subspace of $\Hor_x$. To prove the proposition, we shall verify that for every $x\in \partial E$ we have $$\dim ( \Hor_x\cap \p E ) \geq \dim \Hor_x -1~.$$ We argue by contradiction and assume that there is $x\in \partial E$ such that $\dim ( \Hor_x\cap \p E )\leq  \dim \Hor_x -2$. Using a left-translation, we can assume with no loss of generality that $x=0$. Then let $V$ denote a linear subspace of $\Hor_0$ that is in direct sum with $\Hor_0\cap \p E$. We have $\dim V \geq 2$ and it follows from~\cite[Lemma~3.5]{MR4119259} that $V \cap \partial E$ contains a horizontal line which gives a contradiction and concludes the proof of the proposition.
\end{proof}

We say that a subset of $\frakg$ is measurable to mean that it is $\mu$-measurable where $\mu$ is some, equivalently any, Haar measure on $\frakg$ when seen as an outer measure. The following proposition, ensuring in particular existence of noncharacteristic points in the boundary of non empty measurable precisely monotone strict subsets, will play a key role in the next sections. It is not clear to us whether the measurability assumption can be removed from Proposition~\ref{prop:noncharacteristic-points-exist} and it is the reason that led us to include it in Theorems \ref{thm:main} and \ref{thm:nonfree-case}.

\begin{proposition} \label{prop:noncharacteristic-points-exist}
Let $E\subset \frakg$ be precisely monotone and measurable. Then $\Int(\partial E )= \emptyset$ and $\Nonchar(\p E)$ is a relatively dense subset of $\partial E$.
\end{proposition} 
 
\begin{proof}
We assume that $E\not\in\{\emptyset, \frakg\}$ since otherwise $\partial E = \emptyset$ and there is nothing to prove. We first verify that $\Int(\partial E )= \emptyset$. Recall that both $E$ and $E^c$ are h-convex. Denoting by $\mu$ a Haar mesure on $\frakg$ we get from~\cite[Lemma~6.4]{Rickly06} that there is $c>0$ such that
\begin{equation*} 
\min \left \{\mu(B(x,r) \cap E),\mu(B(x,r) \cap E^c) \right\} \geq c\,\mu(B(x,r))
\end{equation*}
for all $x\in \partial E$ and $r>0$. Here  $B(x,r)$ denotes the open ball with center $x$ and radius $r$ with respect to some given intrinsic metric (we follow here the terminology used in \cite{Rickly06} to which we refer for the definition such metrics). Since $E$ is assumed to be $\mu$-measurable, we know that $\mu$-a.e.~point in $\frakg$  has $\mu$-density 1 for either $E$ or $E^c$ and it follows that $\mu(\partial E)=0$, which implies in turn that $\Int(\partial E) = \emptyset$.

To prove that $\Nonchar(\partial E)$ is relatively dense in $\partial E$, let $\calU \subset \frakg$ be  open and such that $\calU\cap \partial E \not= \emptyset$ and let us prove that $\calU \cap \Nonchar(\partial E) \not= \emptyset$. Using a left-translation we can assume with no loss of generality that $0\in \calU \cap \partial E$. By~\cite[Proposition~5.1]{MontanariMorbidelli20} there is a positive integer $p$ such that the map $\Gamma:(\frakg_1)^p \rightarrow \frakg$ defined by
\begin{equation*}
\Gamma(y_1,\dots,y_p) := y_1 \cdots y_p
\end{equation*}
is open at $0$. Since $\Gamma$ is continuous, there is $\epsilon>0$ such that $\Omega:=\{(y_1,\dots,y_p) \in (\frakg_1)^p:\, \|y_i\| <\epsilon \text{ for } i=1,\dots,p\} \subset \Gamma^{-1}(\calU)$. Here $\|\cdot\|$ denotes some norm on $\frakg_1$. Then there is an open neighborhood $\calU'$ of $0$ in $\frakg$ such that $\calU' \subset \Gamma(\Omega) \subset \calU$. Since $\Int(\partial E) = \emptyset$, one can find $(y_1,\dots,y_p)\in\Omega$ such that $\Gamma(y_1,\dots,y_p) \not\in \partial E$. Since $0\in \partial E$ and $\Gamma(y_1,\dots,y_p) \not\in \partial E$ can be joined by a continuous curve $\gamma \subset \Gamma(\Omega)$ obtained as a concatenation of horizontal segments, we get from Proposition~\ref{tretre} that there is $x \in \gamma$, and therefore $x\in \calU$,  such that $x \in \Nonchar(\partial E)$, which concludes the proof of the proposition.
\end{proof}

For the sake of completeness, we include below a minimizing property of measurable precisely monotone sets. We refer to \cite[Section~3.5]{MR3587666} and the references therein for the notion of intrinsic perimeter that gives an analogue of the classical perimeter in Euclidean spaces. 
 
\begin{proposition}  \label{prop:minimizers} Let $E\subset \frakg$ be precisely monotone and measurable. Then $E$ has locally finite intrinsic perimeter. Furthermore $E$ is a local minimizer for the intrinsic perimeter, which means that for any open set $\Omega \subset \frakg$ such that $\Per(E,\Omega) < +\infty$  we have $\Per(E,\Omega) \leq \Per(F,\Omega)$ for any measurable set $F \subset \frakg$ such that $E \triangle F \Subset \Omega$.
\end{proposition}

\begin{proof}
The fact that $E$ has locally finite intrinsic perimeter whenever $E$ is precisely monotone and measurable follows from~\cite[Theorem~5.6]{Ricklyphd} since precisely monotone sets are h-convex. Then the fact that a measurable precisely monotone set is a local minimizer for the intrinsic perimeter follows from the kinematic formula that relates the intrinsic perimeter to perimeter on horizontal lines, see \cite{MR2165404}. We omit the proof that can be done imitating the proof of \cite[Proposition~3.9]{young2021areaminimizing} that can be verbatim extended to our more general setting, noting that the convexity assumption on $\Omega$ can easily be relaxed.
\end{proof}

We recall now the notion of horizontally affine functions that has been introduced in \cite{LeDonneMorbidelliRigot1} and to which we refer for an exhaustive study of such a class of functions.

\begin{definition} \label{def:h-affine-maps} We say that $\phi :\frakg \rightarrow \R$ is horizontally affine, h-affine in short, if for every $x\in \frakg$, $y\in \frakg_1$, the function $t\in \R \mapsto \phi(x\cdot y^t)$ is affine.
\end{definition} 

Clearly, affine functions on $\frakg$ seen as a vector space are h-affine. However h-affine functions may not be affine. Several equivalent characterizations of step-2 Carnot algebras where h-affine functions are affine can be found in~\cite{LeDonneMorbidelliRigot1}. We recall below a consequence of these characterizations that will be  the only result about h-affine functions needed in the present paper.

\begin{theorem}[{\cite[Theorems~3.2,~1.2,~1.4]{LeDonneMorbidelliRigot1}}] \label{thm:h-affine-maps}
If $\frakg$ is a Heisenberg Carnot algebra or a step-2 rank-3 Carnot algebra that is not isomorphic to the free step-2 rank-3 Carnot algebra then h-affine functions on $\frakg$ are affine.
\end{theorem}

Clearly, sublevel sets of h-affine functions are precisely monotone. As explained in Section~\ref{sect:introduction}, we are interested in the present paper in classifying all precisely monotone subsets of a given step-2 Carnot algebra using sublevel sets of h-affine functions. Our main result Theorem~\ref{thm:main} concerns the case of the free step-2 rank-3 Carnot algebra $\free$ that can be realized as follows.

Given $k\in \{1,2,3\}$ we denote by $\Lambda^k \R^3$ the set of alternating $k$-multilinear forms over $\R^3$ and we set $\Lambda^0 \R^3 := \R$. The free step-2 rank-3 Carnot algebra is given by $$\free:= \Lone \oplus \Ltwo$$ equipped with the Lie bracket for which the only non trivial relations are given by
\begin{equation*} 
[\theta,\tau] := \theta \wedge \tau \quad \text{ for }\theta, \tau \in \Lone
\end{equation*} 
and with the induced group law
\begin{equation*}
(\theta+\omega) \cdot (\tau+\zeta) := \theta + \tau + \omega + \zeta + \theta \wedge \tau \quad \text{ for }\theta, \tau \in \Lone, \, \omega,\zeta \in \Ltwo~.
\end{equation*}

Given $\nu\in \Lambda^3 \R^3 \setminus \{0\}$ it can easily be verified from the very definitions that if $\phi:\free \rightarrow \R$ is given by~\eqref{e:phi-introduction} for some $(\eta_0, \eta_1,\eta_2,\eta_3)\in \Lambda^0\R^3\times\Lambda^1\R^3\times\Lambda^2\R^3 \times \Lambda^3 \R^3$ then $\phi$ is h-affine. Although we will not need the following fact in the present paper, let us mention that it has been proved in~\cite[Theorem~1.1]{LeDonneMorbidelliRigot1} that all h-affine functions on $\free$ are of this form.

\section{Local description in the free step-2 rank-3 Carnot algebra} \label{sect:local-statement-free-case} 

Our first step towards the proof of Theorem~\ref{thm:main} is the following local description near noncharacteristic points.

\begin{proposition} \label{prop:local}
Let $E\subset \free$ be precisely monotone and $x\in \Nonchar(\partial E)$. Then there is an open neighborhood $\calU_x$ of $x$ and there is a non constant h-affine function $\phi_x:\free \rightarrow \R$ such that 
\begin{equation}\label{457} 
\left\{
\begin{aligned}
\calU_x \cap \Int(E) &=\{y \in \calU_x : \phi_x(y) < 0 \} \\
\calU_x\cap\p E &= \calU_x \cap S_x \\
\calU_x\cap\Int(E^c) &=\{y\in \calU_x : \phi_x(y) > 0 \}~.
\end{aligned}
\right.
\end{equation}
where $S_x := \{y \in \free : \phi_x(y) = 0 \}$.
\end{proposition}

This section is devoted to the proof of Proposition~\ref{prop:local}. For notational convenience we will throughout this section identify $\free$ with $\Lambda^1\R^3 \times \Lambda^2\R^3$ and write elements in $\free$ as $x=(\theta,\omega) \in \Lambda^1\R^3 \times \Lambda^2\R^3$. Given a basis $(e_1,e_2,e_3)$ of $\Lone$, we set $e_{ij} := e_i\wedge e_j$ for $1\leq i <j \leq 3$ so that $(e_{12},e_{13},e_{23})$ is a basis of $\Ltwo$. We shall use coordinates in these bases, writing $\theta = \theta_1 e_1 + \theta_2 e_2 + \theta_3 e_3$ and $\omega = \omega_{12} e_{12} + \omega_{13} e_{13} + \omega_{23} e_{23}$ with $\theta_i, \omega_{ij} \in \R$. We denote by $\langle \cdot , \cdot \rangle$ the scalar product on $\Lone$ that makes $(e_1,e_2,e_3)$ an orthonormal basis and we set $e_1^\perp:=\Span\{e_2,e_3\}$.

From now on in this section, we let $E\notin\{\varnothing,\free\}$ denote a precisely monotone subset of $\free$ and, using a left-translation, we assume with no loss of generality that $x=(0,0) \in \Nonchar(\partial E)$.

\subsection{The boundary as a graph near noncharacteristic points}
Since we have $(0,0) \in \Nonchar(\partial E)$, one can find $e_1 \in \Lone \setminus \{0\}$ such that $(e_1,0) \in \Int(E)$. We show in this section, see Proposition~\ref{prop:graph}, that for any choice of such an $e_1$ and any choice of $e_2,e_3 \in \Lone$ so that $(e_1,e_2,e_3)$ is a basis of $\Lone$, one can write $\partial E$ as a graph over $e_1^\perp \times \Ltwo$ near the origin. More importantly, we also get information about the structure of the graph function, see~\eqref{e:graph-f}, that will play a key role later on.

\begin{proposition} \label{prop:graph}
Let $(e_1,e_2,e_3)$ be a basis of $\Lone$ such that $(e_1,0) \in \Int(E)$. There is $\delta >0$ and there are continous functions $A_i: (-\delta,\delta) \rightarrow \R$, $i=2,3$, $B_i : (-\delta,\delta) \rightarrow \R$, $i=1,2,3$, and $C: (-\delta,\delta) \rightarrow \R$ such that the following holds true. Set $W:=\{(\tau,\zeta) \in e_1^\perp \times \Ltwo:\,|\tau_i| < \delta, i=2,3, \, |\zeta_{ij}| < \delta, 1\leq i < j \leq 3\}$ and  $O:=\{(se_1 + \tau,\zeta) \in \free: \, s\in (-1,1),\, (\tau,\zeta) \in W\}$. Then 
\begin{equation}\label{e:local-graph-1} 
\left\{
\begin{aligned}
&O\cap \p E  = \left\{( f(\tau,\zeta) e_1 + \tau   ,\zeta ) \in O :\, (\tau ,\zeta)\in W\right\} \\
&O \cap\Int(E)= \left\{( s e_1 + \tau ,\zeta ) \in O : \, (\tau ,\zeta)\in W,\,  f(\tau ,\zeta)<s<1 \right\} \\
&O\cap\Int(E^c)= \left\{( s e_1 + \tau  ,\zeta ) \in O :\, (\tau ,\zeta)\in W,\, -1<s< f(\tau ,\zeta) \right\}
\end{aligned}
\right.
\end{equation}
where $f : W \rightarrow (-1,1)$ is given by
\begin{multline} \label{e:graph-f}
f(\tau , \zeta): =  A_3(\zeta_{23}) \tau _2 - A_2(\zeta_{23}) \tau _3  - B_1(\zeta_{23}) - B_3(\zeta_{23}) \zeta_{12} + B_2(\zeta_{23}) \zeta_{13} \\
 + C(\zeta_{23})  \tau _2 \zeta_{13} - C(\zeta_{23}) \tau _3 \zeta_{12}~. 
\end{multline}
\end{proposition}

Our starting point to prove Proposition~\ref{prop:graph} is given by~\cite[Theorem~3.7]{MR4119259} from which we know that near the origin $\partial E$ is a so-called intrinsic graph. Namely, set $\ell_1:=\{(se_1,0)\in \free: \, s\in \R\}$ and $\ell_1^+ :=\{(se_1,0)\in \free: \, s>0\}$, $\ell_1^- :=\{(se_1,0)\in \free: \, s<0\}$. By~\cite[Theorem~3.7]{MR4119259} we know that there is $\epsilon >0$ such that, setting  $U:=\{(\theta,\omega) \in  e_1^\perp \times \Ltwo : \, |\theta_2|, |\theta_3|, |\omega_{ij}| < \epsilon, 1\leq i < j \leq 3\}$ and $ \Theta  := U \cdot \ell_1$, the following holds true. There is a continuous function $g:U\to\R$ such that $g(0,0)=0$ and
\begin{equation}\label{e:local-intrinsic-graph} 
\left\{
\begin{aligned}
&\Theta\cap \p E  = \left\{( \theta , \omega ) \cdot (g(\theta , \omega)e_1,0) : (\theta ,\omega  )\in U\right\} \\
&\Theta\cap\Int(E)= ( \Theta\cap \p E) \cdot \ell_1^+ \\
&\Theta\cap\Int(E^c)=( \Theta\cap \p E) \cdot \ell_1^-~.
\end{aligned}
\right.
\end{equation}

We first use Theorem~\ref{thm:CK} and Proposition~\ref{foglietto} to get information about the structure of the map $g$ together with a set constructed from $\Theta\cap \p E$ that is contained in $\partial E$.

\begin{lemma} \label{lem:the-function-g}
There are continuous functions $q_0, q_2,q_3: \{\theta \in e_1^\perp  : \, |\theta_2|, |\theta_3| < \epsilon\} \times (-\epsilon,\epsilon) \rightarrow \R$ such that  
\begin{equation} \label{e:g}
g(\theta,\omega) = q_0(\theta,\omega_{23}) + q_2(\theta,\omega_{23}) \omega_{12} + q_3(\theta,\omega_{23})  \omega_{13}
\end{equation}
for all $(\theta,\omega) \in U$ and there is a map $n:U \rightarrow e_1^\perp$ such that
\begin{multline} \label{e:larger-set-inside-the-boundary}
(\theta,\omega) \cdot (g(\theta,\omega)e_1,0) \cdot  \{(\langle n(\theta,\omega),\xi\rangle + \langle q(\theta,\omega_{23}),\xi'\rangle)e_1 + \xi, e_1\wedge\xi') \in\free: \\ \, \xi ,\xi' \in e_1^\perp, \xi\wedge \xi' =0\} \subset \partial E~.
\end{multline}
for all $(\theta,\omega) \in U$, where $q(\theta,\omega_{23}) := q_2(\theta,\omega_{23}) e_2 + q_3(\theta,\omega_{23}) e_3$.
\end{lemma}
 
\begin{proof}
Let $(\theta,\omega) \in U$ and set $x:=( \theta , \omega \big) \cdot\big(g(\theta , \omega)e_1,0)$. First, note that we know from~\eqref{e:local-intrinsic-graph} that $x\in \Nonchar(\partial E)$ with $x\cdot (e_1,0) \not \in \partial E$. By Proposition~\ref{foglietto}, it follows that $\Hor_x \cap \partial E$ is a 2-dimensional affine subspace of $\Hor_x$ that does not contain $x\cdot (e_1,0)$, i.e., there is $n(\theta,\omega) \in e_1^\perp$ such that 
\begin{equation} \label{e:hor-boundary}
\Hor_x \cap \partial E = x \cdot \{ (\langle n(\theta,\omega), \xi \rangle e_1 + \xi, 0) \in \free:\, \xi \in e_1^\perp \}~.
\end{equation}

Next, set $F:= x^{-1} \cdot E$. Let $\xi \in e_1^\perp \setminus \{0\}$ and let $\frakh:=\Span\{e_1,\xi\} \times \Span\{e_1 \wedge \xi\}$ denote the Lie subalgebra of $\free$ generated by $e_1$ and $\xi$. Then $F \cap \frakh$ is a precisely monotone subset of the Heisenberg algebra $\frakh$. Furthermore, we know from~\eqref{e:local-intrinsic-graph} that
\begin{equation*}
\ell_1^+ \subset  \Int(F) \cap \frakh \subset \Int_\frakh (F \cap \frakh) \quad \text{and} \quad 
\ell_1^- \subset  \Int(F^c) \cap \frakh \subset \Int_\frakh (F^c \cap \frakh)~.
\end{equation*}
Therefore $F \cap \frakh \not \in \{\emptyset,\frakh\}$. By Theorem~\ref{thm:CK}, it follows that $\partial_\frakh(F\cap \frakh)$ is a 2-dimensional linear subspace of $\frakh$ that does not contain $\ell_1$ and $\Int_\frakh (F \cap \frakh)$ and $\Int_\frakh (F^c \cap \frakh)$ are the open half-spaces in $\frakh$ bounded by $\partial_\frakh(F\cap \frakh)$. Since $\Int_\frakh (F \cap \frakh) \subset F$ and  $\Int_\frakh (F^c \cap \frakh) \subset F^c$, it follows that $ \partial_\frakh(F\cap \frakh) \subset \partial F$ and therefore there are $ \alpha_{\theta,\omega,\xi}, \beta_{\theta,\omega,\xi} \in \R$ such that 
\begin{equation} \label{e:alpha-beta}
x \cdot \left\{ ((\alpha_{\theta,\omega,\xi} \,s + \beta_{\theta,\omega,\xi} \, t) e_1 + s \xi ,  t e_1\wedge \xi) \in \free:\, s,t \in \R\right\} \subset \partial E~.
\end{equation}
Then it follows from~\eqref{e:hor-boundary} that 
\begin{equation} \label{e:alpha}
\alpha_{\theta,\omega,\xi} = \langle n(\theta,\omega), \xi \rangle~.
\end{equation}

We have 
\begin{equation*}
x\cdot (\beta_{\theta,\omega,\xi} \, t e_1,t e_1\wedge \xi) = (\theta,\omega + t e_1\wedge \xi) \cdot ((g(\theta,\omega) + \beta_{\theta,\omega,\xi} \, t)e_1,0)~.
\end{equation*}
Since any point in $\free$ can be uniquely written as $x' \cdot y'$ with $x' \in e_1^\perp \times \Ltwo$ and $y'\in \ell_1$, it follows from~\eqref{e:local-intrinsic-graph} and~\eqref{e:alpha-beta} that for every $t\in\R$ small enough so that $(\theta,\omega + t e_1\wedge \xi) \in U$ we have
\begin{equation} \label{e:g-affine}
g(\theta,\omega + t e_1\wedge \xi) = g(\theta,\omega) + \beta_{\theta,\omega,\xi} \, t~.
\end{equation}
This implies that there are functions $q_0, q_2,q_3: \{\theta \in e_1^\perp  : \, |\theta_2|, |\theta_3| < \epsilon\} \times (-\epsilon,\epsilon) \rightarrow \R$ such that~\eqref{e:g} holds true. Since $g$ is continuous, we also get that the functions $q_0, q_2,q_3$ are continuous. Going back to~\eqref{e:g-affine}, we get that for all $(\theta,\omega) \in U$, all $\xi \in e_1^\perp$, and all $t\in \R$ small enough,
\begin{equation*}
\begin{split}
g(\theta,\omega) + \beta_{\theta,\omega,\xi} \, t &= q_0(\theta,\omega_{23}) + q_2(\theta,\omega_{23}) (\omega_{12}+ t \xi_2) + q_3(\theta,\omega_{23})  (\omega_{13}+t \xi_3)\\
&= g(\theta,\omega) + t \langle q(\theta,\omega_{23}) , \xi \rangle
\end{split}
\end{equation*}
where $q(\theta,\omega_{23}) := q_2(\theta,\omega_{23}) e_2 + q_3(\theta,\omega_{23}) e_3$. Therefore 
\begin{equation} \label{e:beta}
\beta_{\theta,\omega,\xi} = \langle q(\theta,\omega_{23}) , \xi \rangle
\end{equation}
and~\eqref{e:larger-set-inside-the-boundary} follows from~\eqref{e:alpha-beta} together~\eqref{e:alpha} and~\eqref{e:beta}. 
\end{proof}

\begin{lemma} \label{lem:definition-function-f}
There is an open neighborhood $V\subset U$ of the origin in $e_1^\perp \times \Ltwo$ such that for every $(\tau ,\zeta) \in V$ there is a unique $s\in (-1,1)$ such that 
\begin{equation} \label{e:equation-def-f}
s= g(\tau , \zeta - s\tau \wedge e_1)~.
\end{equation}
Namely, denoting by $f(\tau ,\zeta) \in (-1,1)$ the unique solution of~\eqref{e:equation-def-f}, we have
\begin{equation} \label{e:def-f}
f(\tau ,\zeta) = \frac{g(\tau ,\zeta)}{1-\langle q(\tau ,\zeta_{23}),\tau \rangle}~.
\end{equation}
\end{lemma}

\begin{proof}
The lemma is a straightforward consequence of~\eqref{e:g} letting $V$ be a small enough open neighborhhod of the origin in $e_1^\perp \times \Ltwo$ choosen in such a way that for all $(\tau ,\zeta) \in V$ we have $(\tau , \zeta - s\tau \wedge e_1) \in U$ for all $s\in (-1,1)$, $1-\langle q(\tau ,\zeta_{23}),\tau \rangle \not=0$, and $(1-\langle q(\tau ,\zeta_{23}),\tau \rangle)^{-1} g(\tau ,\zeta)\in (-1,1)$. Note that such a $V$ does exist by continuity of the functions $g$ and $q$.
\end{proof}

\begin{lemma} \label{lem:change-of-coordinates}
There are open neighborhoods $U'\subset U$ and $V'\subset V$ of the origin in $e_1^\perp \times \Ltwo$ such that the map $\Gamma: U' \rightarrow V'$ defined by $\Gamma(\theta,\omega):= (\theta,\omega+g(\theta,\omega) \theta\wedge e_1)$ is a homeomorphism from $U'$ to $V'=\Gamma(U')$. Furthermore we have $g=f\circ \Gamma$ on $U'$. 
\end{lemma}

\begin{proof}
We know that the map $\Gamma$ given by $\Gamma(\theta,\omega):= (\theta,\omega+g(\theta,\omega) \theta\wedge e_1)$ is well-defined and continuous on $U$. We let $U'\subset U$ be a small enough open neighborhood of the origin in $e_1^\perp \times \Ltwo$ choosen in such a way that for all $(\theta,\omega) \in U'$ we have $g(\theta,\omega) \in (-1,1)$ and $\Gamma(\theta,\omega) \in V$ where $V$ is given by Lemma~\ref{lem:definition-function-f}. Let $(\theta,\omega) \in U'$ and set $(\tau ,\zeta) := \Gamma(\theta,\omega)$. We have $g(\theta,\omega) \in (-1,1)$ and $g(\theta,\omega) = g(\tau ,\zeta - g(\theta,\omega) \tau \wedge e_1)$. Therefore it follows from Lemma~\ref{lem:definition-function-f} that $g(\theta,\omega) = f(\tau ,\zeta)$. In other words, we have $g= f\circ \Gamma$ on $U'$. 

To conclude the proof of the lemma, let us verify that $\Gamma:U' \rightarrow e_1^\perp \times \Ltwo$ is injective. Let $(\theta,\omega), (\theta',\omega')\in U'$ be such that $\Gamma(\theta,\omega) = \Gamma(\theta',\omega')$. On the one hand, by definition of $\Gamma$, we have $\theta = \theta'$ and $\omega+g(\theta,\omega) \theta\wedge e_1 = \omega'+g(\theta',\omega')  \theta'\wedge e_1$. On the other hand, since $g=f\circ \Gamma$ on $U'$, we have $g(\theta,\omega) = g(\theta',\omega')$ and all together it follows that $(\theta,\omega)= (\theta',\omega')$. Therefore $\Gamma:U' \rightarrow e_1^\perp \times \Ltwo$ is a continuous and injective map which implies that $V':=\Gamma(U')$ is an open neighborhood of the origin in $e_1^\perp \times \Ltwo$ and $\Gamma:U' \rightarrow V'$ is a homeomorphism.
\end{proof}

We now use the change of variables provided by Lemma~\ref{lem:change-of-coordinates} to write $\partial E$ as a graph, in the usual sense, over $e_1^\perp \times \Ltwo$ in a neighborhood of the origin. We stress that in general it is not true that an intrinsic graph can be written as a standard graph, see for instance \cite[Section~4.1]{ArenaCarusoMonti12}.

\begin{lemma} \label{lem:local-euclidean-graph}
There is an open neighborhood $V''\subset V'$ of the origin in $e_1^\perp \times \Ltwo$ such that setting $\Omega:=\{(se_1 + \tau ,\zeta) \in \free: \, s\in (-1,1),\, (\tau ,\zeta) \in V''\}$, we have 
\begin{equation}\label{e:local-euclidean-graph} 
\left\{
\begin{aligned}
&\Omega\cap \p E  = \left\{( f(\tau ,\zeta) e_1 + \tau   ,\zeta ) \in \Omega :\, (\tau ,\zeta)\in V''\right\} \\
&\Omega\cap\Int(E)= \left\{( s e_1 + \tau ,\zeta ) \in \Omega : \, (\tau ,\zeta)\in V'',\,  f(\tau ,\zeta)<s<1 \right\} \\
&\Omega\cap\Int(E^c)= \left\{( s e_1 + \tau  ,\zeta ) \in \Omega :\, (\tau ,\zeta)\in V'',\, -1<s< f(\tau ,\zeta) \right\}
\end{aligned}
\right.
\end{equation}
where $f:V'' \rightarrow (-1,1)$ is given by Lemma~\ref{lem:definition-function-f}. Furthermore, setting $m:= n \circ \Gamma^{-1}:V'' \rightarrow e_1^\perp$ where $n$ is given by Lemma~\ref{lem:the-function-g}, we have
\begin{multline} \label{e:larger-set-bis}
(f(\tau ,\zeta)e_1 + \tau , \zeta) \cdot \{\left(\left(\langle m(\tau ,\zeta),\xi\rangle+ \langle q(\tau  ,\zeta_{23}), \xi'\rangle\right) e_1 + \xi, e_1 \wedge \xi'\right)\in\free: \\ \, \xi ,\xi' \in e_1^\perp, \xi\wedge \xi' =0 \} \subset \partial E
\end{multline}
for all $(\tau ,\zeta) \in V''$, where the function $q$ is given by Lemma~\ref{lem:the-function-g}.
\end{lemma}

\begin{proof}
Let $(\tau,\zeta)\in V'$ and set $(\theta,\omega) := \Gamma^{-1}(\tau ,\zeta) \in U'$. By Lemma~\ref{lem:change-of-coordinates} we have $(\tau,\zeta) =(\theta, \omega+g(\theta,\omega)\theta\wedge e_1)$ and $f(\tau,\zeta)=g(\theta,\omega)$. Therefore
\begin{equation} \label{e:change-coordinates-at-the-boundary}
(f(\tau ,\zeta)e_1 + \tau , \zeta) = (g(\theta,\omega)e_1+\theta, \omega+g(\theta,\omega)\theta\wedge e_1) = (\theta,\omega) \cdot (g(\theta,\omega)e_1,0)~.
\end{equation}
Then it follows from~\eqref{e:local-intrinsic-graph} that 
\begin{equation} \label{e:first-inclusion}
\left\{( f(\tau ,\zeta) e_1 + \tau   ,\zeta ) \in \free : (\tau ,\zeta)\in V'\right\} \subset \partial E~.
\end{equation}
Now let $V'' \subset V'$ be a small enough open neighborhood of the origin in $e_1^\perp \times \Ltwo$ choosen in such a way that for all $(\tau ,\zeta) \in V''$ we have $(\tau , \zeta - s\tau \wedge e_1) \in U'$ for all $s\in (-1,1)$. Set $\Omega :=\{(se_1 + \tau ,\zeta) \in \free: \, s\in (-1,1),\, (\tau ,\zeta) \in V''\}$ and let us verify that the first line in~\eqref{e:local-euclidean-graph} holds true. Taking into account~\eqref{e:first-inclusion} we only need to verify that 
\begin{equation} \label{e:second-inclusion}
\Omega \cap \partial E \subset \left\{( f(\tau ,\zeta) e_1 + \tau   ,\zeta ) \in \free : (\tau ,\zeta)\in V''\right\}~.
\end{equation}
Let $s\in (-1,1)$, $(\tau ,\zeta) \in V''$, and assume that $(se_1 + \tau ,\zeta) \in \partial E$. By choice of $V''$ we have $(se_1 + \tau ,\zeta) = (\tau , \zeta - s\tau \wedge e_1) \cdot (se_1,0) \in U' \cdot \ell_1$. Therefore it follows from~\eqref{e:local-intrinsic-graph} that there is $(\theta,\omega) \in U'$ such that $(se_1 + \tau ,\zeta) = (\theta,\omega) \cdot (g(\theta,\omega)e_1,0)$. Then it follows from~\eqref{e:change-coordinates-at-the-boundary} that $(se_1 + \tau ,\zeta) = (f(\tau ',\zeta')e_1 + \tau ', \zeta')$      where $(\tau',\zeta') := \Gamma(\theta,\omega)$. This implies in turn that $(\tau,\zeta) = (\tau',\zeta') $ and $s=f(\tau',\zeta')=f(\tau,\zeta)$ which proves~\eqref{e:second-inclusion}.
To conclude the proof of~\eqref{e:local-euclidean-graph}, note that the first line in~\eqref{e:local-euclidean-graph} together with the continuity of $f$ implies that  either $$\left\{( s e_1 + \tau,\zeta ) \in \free :\, (\tau,\zeta)\in V'',\, f(\tau,\zeta)<s<1 \right\} \subset \Int(E)$$ or $$\left\{( s e_1 + \tau,\zeta ) \in \free :\, (\tau,\zeta)\in V'',\, f(\tau,\zeta)<s<1 \right\} \subset \Int(E^c)~.$$
If $\zeta \in \Ltwo$ is such that $(0,\zeta) \in V''$ then $\Gamma^{-1}(0,\zeta) = (0,\zeta)$ and we know from~\eqref{e:local-intrinsic-graph} that for all $s>f(0,\zeta)=g(0,\zeta)$ we have $(se_1,\zeta) = (0,\zeta) \cdot(se_1,0) \in \Int(E)$. Therefore we have
$$\left\{( s e_1 + \tau,\zeta ) \in \free :\, (\tau,\zeta)\in V'',\, f(\tau,\zeta)<s<1 \right\} \subset \Int(E)~.$$
By similar arguments, we also have  $$\left\{( s e_1 + \tau,\zeta ) \in \free :\, (\tau,\zeta)\in V'',\, -1<s<f(\tau,\zeta) \right\} \subset \Int(E^c)~.$$
Recalling the first line of~\eqref{e:local-euclidean-graph}, we finally get that these inclusions are actually equalities which concludes the proof of~\eqref{e:local-euclidean-graph}.

To conclude the proof of the lemma, note that for $(\tau, \zeta) \in V''$ we have $\Gamma^{-1}( \tau,\zeta) = (\tau, \omega)$ where $\omega \in \Ltwo$ is such that $\omega_{23} = \zeta_{23}$. Therefore $q\circ \Gamma^{-1}( \tau,\zeta) = q(\tau,\zeta_{23})$ where $q$ is given by Lemma~\ref{lem:the-function-g}. Letting $m(\tau,\zeta):=n(\Gamma^{-1}(\tau,\zeta))$, we then get~\eqref{e:larger-set-bis} from~\eqref{e:larger-set-inside-the-boundary} and Lemma~\ref{lem:change-of-coordinates}.
\end{proof}

We shall now use~\eqref{e:local-euclidean-graph} and~\eqref{e:larger-set-bis} to get further information about the structure of the function $f$, see Lemma~\ref{lem:F}. We start in the next lemma with a property of the map $m$.

\begin{lemma} \label{lem:the-function-m}
There is $\delta >0$ and there are maps $\widehat{m}_0,\widehat{m}_2,\widehat{m}_3 : (-\delta,\delta) \rightarrow e_1^\perp$ such that 
\begin{equation} \label{e:m-at-0}
m(0,\zeta) = \widehat{m}_0(\zeta_{23}) + \widehat{m}_2(\zeta_{23}) \zeta_{12} + \widehat{m}_3(\zeta_{23}) \zeta_{13}
\end{equation}
for all $\zeta \in \Ltwo$ such that $|\zeta_{12}|, |\zeta_{13}|, |\zeta_{23}|<\delta$.
\end{lemma}

\begin{proof}
We first note that shrinking $U$ if necessary the map $n:U\rightarrow e_I^\perp$ given by Lemma~\ref{lem:the-function-g} is bounded. Indeed otherwise there is a sequence $(\theta_k,\omega_k)$  converging to $(0,0)$ in $e_1^\perp \times \Ltwo$ and such that $\langle n(\theta_k,\omega_k) , n(\theta_k,\omega_k)\rangle$ goes to infinity. Then we get from~\eqref{e:hor-boundary} that
\begin{equation*}
x_k:=(\theta_k,\omega_k) \cdot (g(\theta_k,\omega_k)e_1,0) \cdot (e_1 +\xi_k,0) \in \partial E
\end{equation*}
where $\xi_k :=\langle n(\theta_k,\omega_k) , n(\theta_k,\omega_k)\rangle^{-1} n(\theta_k,\omega_k)$. Since $x_k \rightarrow (e_1,0)$ and $\partial E$ is closed, it follows that $(e_1,0) \in \partial E$ which gives a contradiction. Therefore the map $m = n\circ \Gamma^{-1}$ is bounded as well.

Since $f$ is continuous and $m$ is bounded, one can find $\delta>0$ such that for all $(\xi,\zeta) \in e_1^\perp \times \Ltwo$ such that $|\xi_2|,|\xi_3|<\delta$ and $|\zeta_{12}|, |\zeta_{13}|, |\zeta_{23}|<\delta$, we have $(\xi, \zeta+f(0,\zeta) e_1\wedge \xi) \in V''$ and $f(0,\zeta) + \langle m(0,\zeta), \xi \rangle \in (-1,1)$. For any such $(\xi,\zeta) \in e_1^\perp \times \Ltwo$, we get from~\eqref{e:larger-set-bis} that
\begin{multline*} 
(f(0,\zeta)e_1 , \zeta) \cdot \left(\langle m(0,\zeta),\xi\rangle e_1 + \xi, 0\right)\\
 = \left( \left(f(0,\zeta) + \langle m(0,\zeta), \xi \rangle \right) e_1 + \xi , \zeta + f(0,\zeta) e_1 \wedge \xi  \right) \in  \Omega \cap \partial E
\end{multline*}
and~\eqref{e:local-euclidean-graph} implies in turn that $f(0,\zeta) + \langle m(0,\zeta), \xi \rangle = f(\xi,\zeta+f(0,\zeta) e_1\wedge \xi)$, i.e.,
\begin{equation}  \label{e:m-1}
 \langle m(0,\zeta), \xi \rangle = f(\xi,\zeta+f(0,\zeta) e_1\wedge \xi) - f(0,\zeta)~.
\end{equation}
Setting $\xi^i := 2^{-1}\delta e_i$ for $i=2,3$, we get that 
\begin{equation*}
 m(0,\zeta) = \sum_{i=2}^3 \langle m(0,\zeta), e_i \rangle e_i = 2\delta^{-1} \sum_{i=2}^3 \left( f(\xi^i,\zeta+  f(0,\zeta) e_1 \wedge \xi^i) - f(0,\zeta) \right) e_i
\end{equation*}
and then~\eqref{e:m-at-0} follows from~\eqref{e:def-f} and~\eqref{e:g}.
\end{proof}

We set 
\begin{align*}
F_0(\tau,\zeta_{23})&:= \left(1-\langle q(\tau,\zeta_{23}) , \tau \rangle \right)^{-1} q_0(\tau,\zeta_{23})\\
F_2(\tau,\zeta_{23})&:= \left(1-\langle q(\tau,\zeta_{23}) , \tau \rangle \right)^{-1} q_3(\tau,\zeta_{23})\\
F_3(\tau,\zeta_{23})&:= - \left(1-\langle q(\tau,\zeta_{23}) , \tau \rangle \right)^{-1} q_2(\tau,\zeta_{23})
\end{align*}
where the functions $q_k$ and $q$ are given by Lemma~\ref{lem:the-function-g} so that~\eqref{e:def-f} writes as
\begin{equation} \label{e:f}
f(\tau,\zeta) = F_0(\tau,\zeta_{23})  - F_3(\tau,\zeta_{23})\zeta_{12} + F_2(\tau,\zeta_{23})\zeta_{13}~.
\end{equation}

\begin{lemma} \label{lem:F} Shrinking $\delta$ if necessary, there are continuous functions $A_i: (-\delta,\delta) \rightarrow \R$, $i=2,3$, $B_i : (-\delta,\delta) \rightarrow \R$, $i=1,2,3$, and $C: (-\delta,\delta) \rightarrow \R$ such that
\begin{equation}  \label{e:F}
\left\{
\begin{aligned}
F_0(\tau,\zeta_{23}) &= A_3(\zeta_{23}) \tau_2 - A_2(\zeta_{23}) \tau_3 - B_1(\zeta_{23}) \\
F_2(\tau,\zeta_{23}) &= B_2(\zeta_{23}) + C(\zeta_{23}) \tau_2  \\
F_3(\tau,\zeta_{23}) &= B_3(\zeta_{23}) + C(\zeta_{23}) \tau_3  
\end{aligned}
\right.
\end{equation}
for all $\tau \in e_1^\perp$ such that $|\tau_2|,|\tau_3| < \delta$ and all $\zeta_{23} \in \R$ such that $|\zeta_{23}|<\delta$.
\end{lemma}

\begin{proof} Set $F(\tau,\zeta_{23}):= - F_3(\tau,\zeta_{23}) e_2 + F_2(\tau,\zeta_{23}) e_3$.
Since the function $\zeta_{23} \mapsto  F(0,\zeta_{23})$ is continuous, shrinking $\delta>0$ if necessary, we can assume with no loss of generality that  $1+\langle F(0,\zeta_{23}) , \tau \rangle \not=0$ for all $\tau \in e_1^\perp$ such that $|\tau_2|,|\tau_3| < \delta$ and all $\zeta_{23} \in \R$ such that $|\zeta_{23}|<\delta$. 

We first prove that there are maps $D_0, D_2, D_3 : (-\delta,\delta) \rightarrow e_1^\perp$ such that for $k=0,2,3$, $\tau \in e_1^\perp$ such that $|\tau_2|,|\tau_3| < \delta$ and $\zeta_{23} \in \R$ such that $|\zeta_{23}|<\delta$, 
\begin{equation} \label{e:Ak}
F_k(\tau,\zeta_{23}) = F_k(0,\zeta_{23})+ \langle D_k(\zeta_{23}) , \tau \rangle - F_k(0,\zeta_{23}) R(\tau,\zeta_{23})
\end{equation}
where
\begin{equation*} 
R(\tau,\zeta_{23}) := \frac{- \tau_2 \langle D_3(\zeta_{23}) , \tau \rangle + \tau_3 \langle D_2(\zeta_{23}) , \tau \rangle}{1+\langle F(0,\zeta_{23}) , \tau \rangle}~.
\end{equation*}
In other words we first verify that the functions $F_k$ admit a first-order Taylor expansion with respect to the variable $\tau$. To prove this claim, let $\zeta_{23} \in \R$ be fixed such that $|\zeta_{23}|<\delta$. For $k=0,2,3$, set $F_k(\tau):= F_k(\tau,\zeta_{23})$, $F(\tau):= F(\tau,\zeta_{23})$, $\widehat{F}_k : = F_k(0,\zeta_{23})$, $\widehat{F}:= F(0,\zeta_{23})$, $\widehat{m}_k:= \widehat{m}_k(\zeta_{23})$. On the one hand, we know from~\eqref{e:m-at-0} that for all $\tau \in e_1^\perp$ such that $|\tau_2|,|\tau_3| < \delta$ and $\zeta_{12}, \zeta_{13} \in \R$ such that $|\zeta_{12}|, |\zeta_{13}| < \delta$,
\begin{equation*}
\langle m(0,\zeta),\tau \rangle = \langle \widehat{m}_0,\tau \rangle + \langle \widehat{m}_2, \tau \rangle \zeta_{12} + \langle \widehat{m}_3, \tau \rangle \zeta_{13}~.
\end{equation*}
where $\zeta:= \zeta_{12} e_{12} + \zeta_{13} e_{13} + \zeta_{23} e_{23}$.
On the other hand, routine computions give 
\begin{equation*}
\begin{split}
f(\tau,\zeta+f(0,\zeta) e_1\wedge \tau) - f(0,\zeta) = & \, F_0(\tau)-\widehat{F}_0 \left(1-\langle  F(\tau),\tau \rangle \right)  \\
& +  \left( -(1-\tau_2\widehat{F}_3)F_3(\tau) - \tau_3 \widehat{F}_3 F_2(\tau) + \widehat{F}_3 \right) \zeta_{12} \\
& +  \left( - \tau_2 \widehat{F}_2 F_3(\tau)+(1+\tau_3\widehat{F}_2) F_2(\tau) - \widehat{F}_2 \right) \zeta_{13}~.
\end{split}
\end{equation*}
Then it follows from~\eqref{e:m-1} that 
\begin{equation} \label{e:A0-intermediate}
F_0(\tau)-\widehat{F}_0 \left(1-\langle  F(\tau),\tau \rangle \right) = \langle \widehat{m}_0,\tau \rangle
\end{equation}
and
\begin{equation*}
\begin{cases}
-(1-\tau_2\widehat{F}_3)F_3(\tau) - \tau_3 \widehat{F}_3 F_2(\tau) = - \widehat{F}_3 + \langle \widehat{m}_2, \tau \rangle \\
- \tau_2 \widehat{F}_2 F_3(\tau)+(1+\tau_3\widehat{F}_2) F_2(\tau) = \widehat{F}_2 + \langle \widehat{m}_3, \tau \rangle \\
\end{cases}
\end{equation*}
for all $\tau \in e_1^\perp$ such that $|\tau_2|,|\tau_3| < \delta$. We then get from the above linear system 
\begin{gather*}
 F_3(\tau) = \frac{ \widehat{F}_3 - \langle \widehat{m}_2, \tau \rangle - \tau_3 (\widehat{F}_3 \langle \widehat{m}_3, \tau \rangle + \widehat{F}_2 \langle \widehat{m}_2, \tau \rangle )}{1+\langle \widehat{F} , \tau \rangle}\\
F_2(\tau) = \frac{\widehat{F}_2 + \langle \widehat{m}_3, \tau \rangle - \tau_2 ( \widehat{F}_3 \langle \widehat{m}_3, \tau \rangle + \widehat{F}_2 \langle \widehat{m}_2, \tau \rangle )}{1+\langle \widehat{F} , \tau \rangle}
\end{gather*}
and inserting these expressions in~\eqref{e:A0-intermediate} we get
\begin{equation*}
F_0(\tau) = \langle \widehat{m}_0, \tau \rangle + \widehat{F}_0 \, \frac{1 - (\tau_2 \langle \widehat{m}_2, \tau \rangle + \tau_3 \langle \widehat{m}_3, \tau \rangle)}{1+\langle \widehat{F} , \tau \rangle}~.
\end{equation*}
Then~\eqref{e:Ak} follows setting $D_0:=\widehat{m}_0 - \widehat{F}_0 \widehat{F}$, $D_2:=\widehat{m}_3 - \widehat{F}_2 \widehat{F}$ and $D_3:=- \widehat{m}_2 - \widehat{F}_3 \widehat{F}$.

\smallskip

We next prove that there is a function $C: (-\delta, \delta) \rightarrow \R$ such that
\begin{equation} \label{e:B}
D_2(\zeta_{23}) = C(\zeta_{23}) e_2 \quad \text{and} \quad D_3(\zeta_{23}) =  C(\zeta_{23}) e_3
\end{equation}
for all $\zeta_{23} \in \R$ such that $|\zeta_{23}|<\delta$.  To prove~\eqref{e:B} let $\zeta_{23} \in \R$ be fixed such that $|\zeta_{23}|<\delta$. Using the same notational conventions as before, omitting the dependence on $\zeta_{23}$, we have $f(0,\zeta_{23} e_{23})= \widehat{F}_0$. By~\eqref{e:m-at-0} we have $m(0,\zeta_{23} e_{23}) =\widehat{m}_0$. Setting $\hat{q}:= q(0,\zeta_{23})$, we then get from~\eqref{e:larger-set-bis} that
\begin{multline*}
(\widehat{F}_0 e_1,\zeta_{23}e_{23}) \cdot \left(\left[s\langle \widehat{m}_0 , \xi \rangle + t \langle \hat{q} , \xi \rangle \right] e_1 + s\xi , t e_1 \wedge \xi\right) \\
= \left(\left[\widehat{F}_0 + s\langle \widehat{m}_0 , \xi \rangle + t \langle \hat{q} , \xi \rangle \right] e_1 + s \xi ,\, \zeta_{23} e_{23} + \left(t+ \widehat{F}_0 s\right) e_1 \wedge \xi \right) \in \partial E
\end{multline*}
for all $\xi\in e_1^\perp$ and all $s,t \in \R$. Then~\eqref{e:local-euclidean-graph} implies that for all $\xi\in e_1^\perp$ and all $s,t \in \R$ small enough 
\begin{align*}
\widehat{F}_0 + s\langle \widehat{m}_0 , \xi \rangle + t \langle \hat{q}, \xi \rangle &= f(s \xi , \zeta_{23} e_{23} + (t+ \widehat{F}_0 s) e_1 \wedge \xi ) 
\\&= 
F_0(s\xi)-F_3(s\xi)(t+\widehat F_0 s)\xi_2+ F_2(s\xi)(t+\widehat F_0 s)\xi_3
\\
&=F_0(s\xi) + \widehat{F}_0 s \langle F(s\xi), \xi \rangle  + t \langle F(s\xi), \xi \rangle~.
\end{align*}
It follows that for all $\xi\in e_1^\perp$ and all $s\in \R$ small enough, $\langle F(s\xi), \xi \rangle = \langle \hat{q}, \xi \rangle$. In other words, the function $s \mapsto\langle F(s\xi), \xi \rangle$ is constant and therefore $\langle \widehat{F}, \xi \rangle = \langle F(s\xi), \xi \rangle$ for all $s \in \R$ small enough. Since we have from~\eqref{e:Ak}
\begin{equation*}
\langle F(s\xi), \xi \rangle = \langle \widehat{F} , \xi \rangle + (- \langle D_3 , \xi \rangle \xi_2 + \langle D_2 , \xi \rangle \xi_3) s  - \langle \widehat{F} ,  \xi \rangle \frac{-\langle D_3 , \xi \rangle \xi_2 + \langle D_2 , \xi \rangle \xi_3}{1+s \langle \widehat{F} , \xi \rangle}  s^2~,
\end{equation*}
we finally get that 
\begin{equation*}
0 = \left(1 - \frac{s \langle \widehat{F} , \xi \rangle}{1+s \langle \widehat{F} , \xi \rangle} \right) (-\langle D_3 , \xi \rangle \xi_2 + \langle D_2 , \xi \rangle \xi_3)  = \frac{- \langle D_3 , \xi \rangle \xi_2 + \langle D_2 , \xi \rangle \xi_3}{1+s \langle \widehat{F} , \xi \rangle}
\end{equation*}
for all $\xi \in e_1^\perp$ which implies~\eqref{e:B}.

To conclude the proof of the lemma, note that~\eqref{e:B} implies that $R(\tau, \zeta_{23}) =0$. Then we set $A_2(\zeta_{23}) := - \langle D_0(\zeta_{23}), e_3 \rangle$, $A_3(\zeta_{23}) := \langle D_0(\zeta_{23}), e_2 \rangle$, $B_1(\zeta_{23}) := - F_0(0,\zeta_{23})$, $B_2(\zeta_{23}) := F_2(0,\zeta_{23})$, and $B_3(\zeta_{23}) := F_3(0,\zeta_{23})$ to get \eqref{e:F}. Note that the continuity of the functions $A_i$, $B_i$ and $C$ follows from the continuity of the functions $F_i$.
\end{proof}

The proof of Proposition~\ref{prop:graph} is now complete. Indeed \eqref{e:local-graph-1} follows from \eqref{e:local-euclidean-graph} and \eqref{e:graph-f} from \eqref{e:f} and \eqref{e:F}.

\subsection{The boundary as a level set of a h-affine function near noncharacteristic points}

To complete the proof of Proposition~\ref{prop:local} we choose in this section a basis $(e_1,e_2,e_3)$ of $\Lone$ such that $(e_i,0) \in \Int(E)$ for $i=1,2,3$. Such a basis does exist. Indeed, recall that $(0,0)\in \Nonchar(\partial E)$. Therefore one can find $e_1 \in \Lone \setminus\{0\}$ such that $(e_1,0) \in \Int(E)$. One can then choose $e_2,e_3$ close enough to $e_1$ in such a way that $(e_i,0) \in \Int(E)$ for $i=2,3$ and such that $e_1,e_2,e_3$ are linearly independent. We then get from Proposition~\ref{prop:graph} that near the origin $\partial E$ can be written as a graph over $e_1^\perp \times \Ltwo$ as well as a graph over $e_2^\perp \times \Ltwo$ and $e_3^\perp \times \Ltwo$ and each one of the graph functions has a structure given by~\eqref{e:graph-f} as explicitly stated in the next proposition.

\begin{proposition} \label{prop:graphs} There is $\delta >0$ and there are continuous functions $A^i_j : (-\delta,\delta) \rightarrow \R$, $i,j = 1,2,3$, $i\not=j$, $B^i_j : (-\delta,\delta) \rightarrow \R$, $i,j = 1,2,3$, $C^i: (-\delta,\delta) \rightarrow \R$, $i= 1,2,3$, such that the following holds true. Set $\calU :=\{(\tau,\zeta) \in \free : \, |\tau_i| < \delta, \, i=1,2,3, \, |\zeta_{ij}| < \delta, \, 1\leq i <j \leq 3\}$. For every $(\tau,\zeta) \in \calU$, the following four conditions are equivalent~: 
\begin{subequations}\begin{align}
& \phantom{ab} (\tau,\zeta) \in \partial E~,\label{e:in-the-boundary}\\
&\begin{split}
\tau_1 - A^1_3(\zeta_{23}) \tau_2 + A^1_2(\zeta_{23}) \tau_3  + B^1_1(\zeta_{23}) &+ B^1_3(\zeta_{23}) \zeta_{12} - B^1_2(\zeta_{23}) \zeta_{13}\\
& - C^1(\zeta_{23})  \tau_2 \zeta_{13} + C^1(\zeta_{23}) \tau_3 \zeta_{12} = 0~,\end{split}
  \label{e:graph-1}\\
&\begin{split}
 \tau_2 - A^2_1(\zeta_{13}) \tau_3 + A^2_3(\zeta_{13}) \tau_1 + B^2_2(\zeta_{13}) &+ B^2_1(\zeta_{13}) \zeta_{23} + B^2_3(\zeta_{13}) \zeta_{12} \\
 & + C^2(\zeta_{13})  \tau_3 \zeta_{12} + C^2(\zeta_{13}) \tau_1 \zeta_{23} = 0~,\end{split}
   \label{e:graph-2}\\
&\begin{split}
 \tau_3 - A^3_2(\zeta_{12}) \tau_1 + A^3_1(\zeta_{12}) \tau_2 + B^3_3(\zeta_{12}) &- B^3_2(\zeta_{12}) \zeta_{13} + B^3_1(\zeta_{12}) \zeta_{23}\\
 & + C^3(\zeta_{12})  \tau_1 \zeta_{23} - C^3(\zeta_{12}) \tau_ 2\zeta_{13} = 0~. \end{split}
  \label{e:graph-3}
\end{align}\end{subequations}
\end{proposition}

Recall that since $(0,0) \in \calU \cap \partial E$ we have
\begin{equation} \label{e:B^i_i(0)}
B^1_1(0) = B^2_2(0) = B^3_3(0) =0~.
\end{equation}

Recall also that $(0,0) \in \Nonchar(\partial E)$. Therefore we know from Proposition~\ref{foglietto} that $\Hor_{(0,0)} \cap \partial E = (\Lone \times \{0\}) \cap \partial E$ is a 2-dimensional linear subspace of $\Lone \times \{0\}$. By choice of $e_1,e_2,e_3$, we also have $(e_i,0) \not \in \partial E$ for $i=1,2,3$. Therefore there are $a_{12}, a_{13} \in \R \setminus\{0\}$ and $a_{23}<0$ such that 
\begin{equation*}
(\Lone \times \{0\}) \cap \partial E = \{(\tau,0) \in \free:\, a_{23} \tau_1 - a_{13} \tau_2 + a_{12} \tau_3 =0\}~.
\end{equation*}
We also know from Proposition~\ref{prop:graphs} and~\eqref{e:B^i_i(0)} that 
\begin{equation*}
\begin{split}
(\Lone \times \{0\}) \cap \calU \cap \partial E  &= \{(\tau,0) \in \calU:\, \tau_1 - A^1_3(0) \tau_2 + A^1_2(0) \tau_3 =0\} \\
&= \{(\tau,0) \in \calU:\, \tau_2 - A^2_1(0) \tau_3 + A^2_3(0) \tau_1 =0\}  \\
&= \{(\tau,0) \in \calU:\, \tau_3 - A^3_2(0) \tau_1 + A^3_1(0) \tau_2 =0\}~.
\end{split}
\end{equation*} 
Therefore
\begin{equation} \label{e:A^i_j(0)}
\left\{
\begin{aligned}
&A^1_3(0) = \frac{a_{13}}{a_{23}}~, \quad A^1_2(0) = \frac{a_{12}}{a_{23}}~,\\
&A^2_1(0) = \frac{a_{12}}{a_{13}}~,\quad A^2_3(0) = - \frac{a_{23}}{a_{13}}~,\\
&A^3_2(0) = -\frac{a_{23}}{a_{12}}~,\quad A^3_1(0) = - \frac{a_{13}}{a_{12}}~.
\end{aligned}
\right.
\end{equation}

\begin{lemma} \label{lem:A^1_3-A^1_2-B^1_1}
There are $c\in \R$ and $b_1\in \R$ such that for all $\zeta_{23}$ small enough
\begin{equation} \label{e:A^1_3-A^1_2-B^1_1}
\left\{
\begin{aligned}
&A^1_2(\zeta_{23}) = \frac{a_{12}}{a_{23} + c \zeta_{23}}~, \\
&A^1_3(\zeta_{23}) = \frac{a_{13}}{a_{23} + c \zeta_{23}}~, \\
&B^1_1(\zeta_{23}) = \frac{b_1\zeta_{23}}{a_{23} + c \zeta_{23}}~.
\end{aligned}
\right.
\end{equation}
\end{lemma}

\begin{proof}
We first look at $(\Span\{e_1\} \times \Span\{e_{23}\}) \cap \calU \cap \partial E$. We know from Proposition~\ref{prop:graphs} and~\eqref{e:B^i_i(0)} that $(\Span\{e_1\} \times \Span\{e_{23}\}) \cap \calU \cap \partial E \not=\emptyset$. Since $A^3_2(0)A^2_3(0)\not=0$, see~\eqref{e:A^i_j(0)}, we also know from $-A^3_2(0) \times$~\eqref{e:graph-2}, $A^2_3(0) \times$~\eqref{e:graph-3} and~\eqref{e:B^i_i(0)} that for all $(\tau_1 e_1,\zeta_{23} e_{23}) \in \calU \cap \partial E$ the following two equivalent conditions hold true~:
\begin{align*}
 - A^3_2(0)  \left[ A^2_3(0) \tau_1 + B^2_1(0) \zeta_{23} + C^2(0) \tau_1 \zeta_{23} \right] &= 0~,\\
A^2_3(0) \left[ - A^3_2(0) \tau_1 +  B^3_1(0) \zeta_{23} + C^3(0)  \tau_1 \zeta_{23} \right] &= 0~. 
 \end{align*}
Set $\phi(\tau,\zeta):=- A^3_2(0)  [ A^2_3(0) \tau_1 + B^2_1(0) \zeta_{23} + C^2(0) \tau_1 \zeta_{23} ]$ and $\psi(\tau,\zeta):= A^2_3(0) [ - A^3_2(0) \tau_1 +  B^3_1(0) \zeta_{23} + C^3(0)  \tau_1 \zeta_{23} ]$. These non constant functions $\phi,\psi: \free \rightarrow \R$ are h-affine with $\phi(0,0) = \psi(0,0)$ and we know from the previous argument that $\{(\tau,\zeta) \in \calU:\, \phi(\tau,\zeta)=0\} = \{(\tau,\zeta) \in \calU:\, \psi(\tau,\zeta)=0\}$. Then it follows from Corollary~\ref{cor:level-sets} that $-A^3_2(0) B^2_1(0) = A^2_3(0) B^3_1(0)$ and $- A^3_2(0) C^2(0) =  A^2_3(0) C^3(0)$. Taking into account \eqref{e:A^i_j(0)} we get that $- a_{13} B^2_1(0) = a_{12} B^3_1(0)$ and $- a_{13}  C^2(0) =  a_{12}C^3(0)$. Using two consecutive permutations of the coordinates, we then get the following relations~: 
\begin{align}
- a_{13} B^2_1(0) = a_{12} B^3_1(0) &=:b_1~, \label{e:b1}\\
 a_{12} B^3_2(0) = a_{23} B^1_2(0)&=:b_2~, \label{e:b2}\\
 a_{23} B^1_3(0) = - a_{13} B^2_3(0) &=:b_3~, \label{e:b3}\\
- a_{13}  C^2(0) =  a_{12} C^3(0) =  a_{23} C^1(0)  &=:c~. \label{e:c}
\end{align}

Next let $\zeta_{23} \in \R$ be fixed small enough so that $A^1_2(\zeta_{23})\not=0$ (recall that $A^1_2(0)\not=0$, see~\eqref{e:A^i_j(0)}, and $A^1_2$ is continous) and so that $(\Lone \times \{\zeta_{23} e_{23}\}) \cap \calU \cap \partial E \not=\emptyset$, see Proposition~\ref{prop:graphs}. Using~\eqref{e:graph-1}, $A^1_2(\zeta_{23})\times$~\eqref{e:graph-3} and~\eqref{e:B^i_i(0)}, we know that for all $\tau \in \Lone$ such that $|\tau_i|<\delta$ for $i=1,2,3$, we have $(\tau,\zeta_{23} e_{23}) \in \partial E$ if and only if the following two equivalent conditions hold true~:
\begin{align*}
\tau_1 - A^1_3(\zeta_{23}) \tau_2 + A^1_2(\zeta_{23}) \tau_3  + B^1_1(\zeta_{23})  &= 0~,\\
 A^1_2(\zeta_{23}) \left[ (- A^3_2(0)+ C^3(0) \zeta_{23}) \tau_1  +  A^3_1(0)\tau_2  +  \tau_3 +  B^3_1(0) \zeta_{23} \right] &= 0~. 
\end{align*}
This implies that 
\begin{align*}
 A^1_2(\zeta_{23})(- A^3_2(0)+ C^3(0) \zeta_{23}) &= 1~,\\
 A^1_2(\zeta_{23})  A^3_1(0) &= - A^1_3(\zeta_{23})~, \\
 A^1_2(\zeta_{23}) B^3_1(0) \zeta_{23} &= B^1_1(\zeta_{23})~.
\end{align*}
and ~\eqref{e:A^1_3-A^1_2-B^1_1} follows taking into account~\eqref{e:A^i_j(0)},~\eqref{e:b1} and~\eqref{e:c}.
\end{proof}

\begin{lemma} For all $\zeta_{23}$ small enough, we have
\begin{equation} \label{e:C^1-B^1_2-B^1-3}
\left\{
\begin{aligned}
&C^1(\zeta_{23}) = \frac{c}{a_{23} +c \zeta_{23}}~,\\
&B^1_2(\zeta_{23}) = \frac{b_2}{a_{23} +c \zeta_{23}}~,\\
& B^1_3(\zeta_{23}) = \frac{b_3}{a_{23} +c \zeta_{23}}~,
\end{aligned}
\right.
\end{equation}
where $c$ is given by Lemma~\ref{lem:A^1_3-A^1_2-B^1_1}, and $b_2$, $b_3$ are given by~\eqref{e:b2} and~\eqref{e:b3}.
\end{lemma}

\begin{proof}
Let $\zeta_{12}, \zeta_{23} \in \R$ be fixed small enough so that $(\Span\{e_2,e_3\} \times \{\zeta_{12} e_{12} + \zeta_{23} e_{23}\}) \cap \calU \cap \partial E \not=\emptyset$, see Proposition~\ref{prop:graphs}, and so that $A^1_3(\zeta_{23}) \not=0$ (recall that $A^1_3(0) \not=0$, see~\eqref{e:A^i_j(0)}, and $A^1_3$ is continuous). Using~\eqref{e:graph-1}, $-A^1_3(\zeta_{23}) \times$~\eqref{e:graph-2} and~\eqref{e:B^i_i(0)}, we know that for all $\tau_2,\tau_3 \in (-\delta,\delta)$ we have $(\tau_2 e_2 + \tau_3 e_3 ,\zeta_{12} e_{12} +\zeta_{23} e_{23}) \in \partial E$ if and only if the following two equivalent conditions hold true~:
\begin{align*}
- A^1_3(\zeta_{23})  \tau_2 + (A^1_2(\zeta_{23}) + C^1(\zeta_{23})  \zeta_{12}) \tau_3  + B^1_1(\zeta_{23}) + B^1_3(\zeta_{23}) \zeta_{12}  &= 0~,\\
- A^1_3(\zeta_{23}) \left[ \tau_2 - (A^2_1(0) - C^2(0)\zeta_{12}) \tau_3  + B^2_1(0) \zeta_{23} + B^2_3(0) \zeta_{12} \right] &=0~.
\end{align*}
Considering the coefficients in front of $\tau_3$, this implies that  $( A^1_2(\zeta_{23}) + C^1(\zeta_{23})  \zeta_{12}) =  A^1_3(\zeta_{23})(A^2_1(0) - C^2(0)\zeta_{12})$. 
Therefore 
\begin{equation*}
 C^1(\zeta_{23})  \zeta_{12} = A^1_3(\zeta_{23})(A^2_1(0) - C^2(0)\zeta_{12}) - A^1_2(\zeta_{23})
\end{equation*}
and the  form  of $C^1(\zeta_{23})$ follows from~\eqref{e:A^i_j(0)},~\eqref{e:A^1_3-A^1_2-B^1_1} and~\eqref{e:c}. Considering the constant terms,  we get that  $B^1_1(\zeta_{23}) + B^1_3(\zeta_{23}) \zeta_{12} = - A^1_3(\zeta_{23}) (B^2_1(0) \zeta_{23} + B^2_3(0) \zeta_{12})$. Therefore
\begin{equation*}
 B^1_3(\zeta_{23}) \zeta_{12} = - A^1_3(\zeta_{23}) ( B^2_1(0) \zeta_{23} + B^2_3(0) \zeta_{12}) -  B^1_1(\zeta_{23})
\end{equation*}
and the form of $B^1_3(\zeta_{23})$ follows from~\eqref{e:A^1_3-A^1_2-B^1_1},~\eqref{e:b1} and~\eqref{e:b3}.

To get the form of $B^1_2(\zeta_{23})$ we argue in a similar way considering $(\Span\{e_3\} \times \{\zeta_{13} e_{13} + \zeta_{23} e_{23}\}) \cap \calU \cap \partial E$. Namely, let $\zeta_{13}, \zeta_{23} \in \R$ be fixed small enough so that $(\Span\{e_3\} \times \{\zeta_{13} e_{13} + \zeta_{23} e_{23}\}) \cap \calU \cap \partial E \not=\emptyset$, see Proposition~\ref{prop:graphs}, and so that $A^1_2(\zeta_{23}) \not=0$ (recall that $A^1_2(0) \not=0$, see~\eqref{e:A^i_j(0)}, and $A^1_2$ is continuous). Using~\eqref{e:graph-1}, $A^1_2(\zeta_{23}) \times$~\eqref{e:graph-3} and~\eqref{e:B^i_i(0)}, we know that for all $\tau_3 \in (-\delta,\delta)$ we have $(\tau_3 e_3,\zeta_{13} e_{13} +\zeta_{23} e_{23}) \in \partial E$ if and only if the following two equivalent conditions hold true~:
\begin{align*}
& A^1_2(\zeta_{23}) \tau_3  + B^1_1(\zeta_{23})  - B^1_2(\zeta_{23}) \zeta_{13}= 0~,\\
& A^1_2 (\zeta_{23}) \left[  \tau_3  - B^3_2(0) \zeta_{13} + B^3_1(0) \zeta_{23} \right] = 0~.
\end{align*}
This implies that $B^1_1(\zeta_{23})  - B^1_2(\zeta_{23}) \zeta_{13} = A^1_2 (\zeta_{23}) (- B^3_2(0) \zeta_{13} + B^3_1(0) \zeta_{23})$. Therefore
\begin{equation*}
B^1_2(\zeta_{23}) \zeta_{13} = B^1_1(\zeta_{23}) + A^1_2 (\zeta_{23}) (B^3_2(0) \zeta_{13} - B^3_1(0) \zeta_{23})
\end{equation*}
and the form of $B^1_2(\zeta_{23})$ follows from~\eqref{e:A^1_3-A^1_2-B^1_1},~\eqref{e:b1} and~\eqref{e:b2}.
\end{proof}

To conclude the proof of Proposition~\ref{prop:local}, we set $\eta_0:= c\in \R = \Lambda^0\R^3$, $\eta_1:=b_1 e_1 + b_2 e_2 + b_3 e_3 \in \Lone$, $\eta_2:= a_{12} e_{12} + a_{13} e_{13} +a_{23} e_{23} \in \Ltwo \setminus \{0\}$ and we let $\phi:\free \rightarrow \R$ be the non constant h-affine function given by
\begin{equation} \label{e:the-h-affine-map}
\phi (\tau, \zeta) \nu := \eta_2 \wedge \tau + \eta_1 \wedge \zeta + \eta_0\, \tau \wedge \zeta 
\end{equation}
where $\nu := e_1 \wedge e_2 \wedge e_3$. By~\eqref{e:A^1_3-A^1_2-B^1_1} and~\eqref{e:C^1-B^1_2-B^1-3} we have for $\zeta_{23}$ small enough
\begin{multline*}
\phi (\tau, \zeta)  = (a_{23} + c \zeta_{23}) (\tau_1 - A^1_3(\zeta_{23}) \tau_2 + A^1_2(\zeta_{23}) \tau_3  + B^1_1(\zeta_{23}) + B^1_3(\zeta_{23}) \zeta_{12}\\ - B^1_2(\zeta_{23}) \zeta_{13}
 - C^1(\zeta_{23})  \tau_2 \zeta_{13} + C^1(\zeta_{23}) \tau_3 \zeta_{12}).
\end{multline*}
Therefore, shrinking $\calU$ if necessary, we get from ~\eqref{e:in-the-boundary} and~\eqref{e:graph-1} that 
\begin{equation*}
\calU \cap\p E = \{(\tau,\zeta) \in \calU:\, \phi (\tau, \zeta) = 0\}~.
\end{equation*}
Since we have choosen $a_{23} < 0$, we also get from the last two lines in~\eqref{e:local-graph-1} that
\begin{align*}
\calU \cap \Int(E) &=\{(\tau,\zeta) \in \calU : \phi(\tau,\zeta) < 0 \} \\
\calU \cap\Int(E^c) &=\{(\tau,\zeta)\in \calU : \phi (\tau,\zeta) > 0 \}
\end{align*}
which concludes the proof of~\eqref{457}.

\section{Classification in the free step-2 rank-3 Carnot algebra} \label{sect:global-statement-free-case}

In this section we upgrade the local statement given by Proposition~\ref{prop:local} into a global one, that will in turn  imply Theorem~\ref{thm:main}. 

We set 
\begin{equation*} \label{e:def-Sigma}
\Sigma :=\bigcup_{\substack{\xi,\tau \in \Lone \\ \xi \wedge \tau \not=0}}\Lie(\xi,\tau)
\end{equation*}
where $\Lie(\xi,\tau) := \Span\{\xi,\tau\} \oplus \Span\{\xi \wedge \tau\}$ denotes the Lie subalgebra of $\free$ generated by $\xi$ and $\tau$. We start with the following preliminary step.

\begin{lemma} \label{lem:step0} Let $E\subset \free$ be precisely monotone and $x\in \Nonchar(\partial E)$. Let $\calU_x$ be an open neighborhood of $x$ and $\phi_x:\free \rightarrow \R$ be a non constant h-affine function for which~\eqref{457} holds true. Then $S_x \cap (x\cdot \Sigma )\subset \partial E$ where $S_x:=\{y\in \free:\, \phi_x(y)=0\}$.
\end{lemma} 

\begin{proof}
Using a left-translation, we can assume with no loss of generality that $x=0$. We set $\calU:= \calU_{0}$, $\phi := \phi_0$, and $S:=S_0$. Let $\xi,\tau\in \Lone$ be given such that $\xi \wedge \tau \not=0$. Set $\frakh:= \Lie(\xi,\tau)$ and let us prove that $S \cap \frakh \subset \partial E$.

The restriction $\phi_{|\frakh}:\frakh \rightarrow \R$ of the function $\phi$ to $\frakh$ is h-affine on the Heisenberg algebra $\frakh$ and such that $\phi_{|\frakh}(0)=0$. Therefore it follows from Theorem~\ref{thm:h-affine-maps} that $\phi_{|\frakh}$ is a linear form on $\frakh$. 

If $\Ker \phi_{|\frakh} = \frakh$, we get from~\eqref{457} that $\calU \cap \frakh \subset \partial E$. We fix $s\not=0$ close enough to $0$ so that $s\xi \wedge \tau \in \calU$ and we consider the horizontal lines  $\ell_1:= \{ \xi^t \in \frakh :\, t\in\R\} \subset \frakh$ and $\ell_2:= (s\xi\wedge \tau ) \cdot \{ \ \tau^t \in \frakh :\, t\in\R\} \subset \frakh$. On the one hand, we have $\calU \cap \ell_j \subset \partial E$ for $j=1,2$ and it follows from Lemma~\ref{linee} that $\ell_1 \cup \ell_2 \subset \partial E$. On the other hand $\ell_1$ and $\ell_2$ are skew lines in $\frakh$ in the sense of~\cite{CheegerKleiner10}. Since they are contained in $\frakh \cap \partial E$, Lemma~\ref{linee} together with~\cite[Lemma~4.10]{CheegerKleiner10} applied to the set $\mathcal{G}:= \frakh \cap \p E$ implies that $\frakh \subset \p E$. 

If $\Ker \phi_{|\frakh} \not= \frakh$ then $\Ker \phi_{|\frakh}$ is a 2-dimensional linear subspace of $\frakh$. If $\Ker \phi_{|\frakh} = \Span\{\xi,\tau\}$ then $\Ker \phi_{|\frakh}$ is a linear subspace of $ \Hor_0$. We also know that $\Hor_0 \cap \partial E$ is a linear subspace of $ \Hor_0$ (see Proposition~\ref{foglietto}) that contains $\calU \cap \Ker \phi_{|\frakh}$ (see~\eqref{457}). This implies that $\Ker \phi_{|\frakh} \subset \Hor_0 \cap \partial E$ and therefore $S \cap \frakh \subset \partial E$. If $\Ker \phi_{|\frakh} \not= \Span\{\xi,\tau\}$, there is $\theta \in \Span\{\xi,\tau\}$ such that $\phi(\theta) \not=0$. Then we get from~\eqref{457} that for all $s>0$ small enough either $s\theta \in \Int(E) \cap \frakh$ and $-s\theta \in \Int(E^c) \cap \frakh$, or, $-s\theta \in \Int(E) \cap \frakh $ and $s\theta \in \Int(E^c) \cap \frakh$. It follows that $E \cap \frakh$ is a precisely monotone subset of $\frakh$ that is neither $\emptyset$ nor $\frakh$. Therefore $\partial_\frakh (E\cap \frakh)$ is 2-dimensional linear subspace of $\frakh$ by Theorem~\ref{thm:CK}. We also know from~\eqref{457} that $\calU \cap \partial_\frakh (E\cap \frakh) \subset \calU \cap \partial E \cap \frakh = \calU \cap  \Ker \phi_{|\frakh}$. This implies that the 2-dimensional linear subspaces $\Ker \phi_{|\frakh}$ and $\partial_\frakh (E\cap \frakh)$ coincide and therefore $S \cap \frakh = \partial_\frakh (E\cap \frakh) \subset \partial E$.
\end{proof}

For the rest of this section, we fix a precisely monotone measurable subset $E$ of $\free$ such that $E\notin\{\emptyset,\free\}$. By Proposition~\ref{prop:noncharacteristic-points-exist} we know that $\Nonchar(\partial E) \not=\emptyset$. Using a left-translation, we can assume with no loss of generality $0 \in \Nonchar(\partial E)$. We set $\calU:= \calU_{0}$, $\phi := \phi_0$, and $S:=\{ x\in \free :\, \phi(x) =0 \}$, where $\calU_{0}$ is an open neighborhood of $0$ and $\phi_0 : \free \rightarrow \R$ is a non constant h-affine function given by Proposition~\ref{prop:local} so that $0 \in \Nonchar(S)$ and
\begin{equation}\label{e:description} 
\left\{
\begin{aligned}
\calU \cap \Int(E) &=\{y \in \calU : \phi(y) <0 \} \\
\calU\cap\p E &=\calU \cap S \\
\calU\cap\Int(E^c) &=\{y\in \calU : \phi(y) >0 \}~.
\end{aligned}
\right.
\end{equation}
Since $\Nonchar(\partial E)$ is a relatively open subset of $\partial E$, shrinking $\calU$ if necessary, we also assume with no loss of generality that
\begin{equation} \label{e:no-characteristic-points}
\calU \cap \Char(\partial E) = \emptyset~.
\end{equation}

The proof of Theorem~\ref{thm:main} will proceed in the following steps:

\smallskip 

\noindent (1) Lemma~\ref{lem:step1}: $S\cap ( x \cdot \Sigma) \subset \partial E$ for all $x\in \Nonchar(\partial E) \cap \calU$, in particular, $S\cap \Sigma \subset \partial E$. 

\noindent (2) Lemma~\ref{lem:step2}: $S\setminus \Sigma \subset \partial E$.

\noindent (3)  Lemma~\ref{lem:step3}:  $\partial E \subset S$.

\noindent (4) Lemma~\ref{lem:step4}: $\Int(E) = \{y\in \free :\, \phi(y) < 0\}$ and $\overline{E}= \{y\in \free :\, \phi(y) \leq 0\}$.

\begin{lemma} \label{lem:step1} We have $S\cap (x \cdot \Sigma)  \subset \partial E$ for all $x\in \Nonchar(\partial E) \cap \calU$. In particular
\begin{equation} \label{e:S-inter-Sigma-in-the boundary}
S\cap  \Sigma \subset \partial E~.
\end{equation}
\end{lemma}

\begin{proof}
If $x\in \Nonchar(\partial E) \cap \calU$ then $\calU_x:=\calU$ and $\phi_x:=\phi$ are an open neighborhood of $x$, respectively, a non constant h-affine function, for which~\eqref{457} holds true. Therefore $S\cap ( x \cdot \Sigma) \subset \partial E$ by Lemma~\ref{lem:step0}.
\end{proof}

We give in Lemma~\ref{lem:step2-1} below a condition on points in $S\setminus\Sigma$ that ensures that they belong to $\partial E$. To prove that $S\setminus\Sigma \subset \partial E$, see Lemma~\ref{lem:step2}, we shall next verify thanks to Lemma~\ref{lem:step2-2} that this condition holds true on a relatively dense subset of $S\setminus \Sigma$.

From now on in this section, we identify $\free$ with $\Lone \times \Ltwo$. We recall for further use that 
\begin{equation} \label{e:equation-Sigma}
\Sigma = \{(\theta,\omega) \in \free:\, \theta \wedge \omega =0\}~.
\end{equation} 

We also recall that given $\nu\in\Lambda^3\R^3 \setminus\{0\}$,  there are $\eta_0 \in \Lambda^0\R^3$, $\eta_1\in\Lone$, and $\eta_2 \in \Ltwo \setminus \{0\}$ such that the function $\phi$ showing up in \eqref{e:description} is given by 
\begin{equation} \label{e:phi}
\phi(\theta,\omega)\nu  = \eta_2 \wedge \theta + \eta_1 \wedge \omega + \eta_0\, \theta \wedge \omega 
\end{equation}
for all $(\theta,\omega) \in \free$, see~\eqref{e:the-h-affine-map}. For $j=1,2$ we denote by $\phi_j : \Lambda^j\R^3 \rightarrow \R$ the linear form defined as the restriction of $\phi$ to $\Lambda^j\R^3$. In other words, $\phi_1 : \Lone \rightarrow \R$ is the non constant linear form on $\Lone$ given by $\phi_1(\theta) := \phi(\theta,0)$ and $\phi_2 : \Ltwo \rightarrow \R$ is the linear form on $\Ltwo$ given by $\phi_2(\omega) := \phi(0,\omega)$.

We recall that given $\omega \in \Ltwo$ the space of exterior annihilators of $\omega$ of order 1 is defined as 
\begin{equation*} \label{e:annihilators}
\Anh (\omega) := \{ \xi\in \Lone:\, \omega \wedge \xi = 0\}~.
\end{equation*}
We also recall that if $\omega \in \Ltwo \setminus \{0\}$ then $\Anh (\omega)$ is a 2-dimensional linear subspace of $\Lone$ and for every $\xi \in \Anh (\omega)\setminus \{0\}$ there is $\tau \in \Anh (\omega)$ such that $\omega = \xi \wedge \tau$. 

\begin{lemma} \label{lem:step2-1} 
Let $(\theta, \omega)\in S\setminus\Sigma$. Assume there are $\xi, \tau \in \Anh(\omega)$ and $p,q \in\R$ such that $\omega = \xi \wedge \tau$,  
\begin{equation}\label{e:step2-1-cdt-1} 
\phi(\xi,\theta \wedge \xi) p + \phi(\tau,\theta \wedge \tau) q= - \phi_2(\omega)~,
\end{equation} 
and such that for any $\epsilon >0$ there are $r\in \R\setminus\{1\}$ and $u,v\in (-\epsilon,\epsilon)$ such that 
\begin{gather} 
 -rqu+rpv =1-r~, \label{e:step2-1-cdt-2} \\
\phi_2(\xi \wedge \theta) u+\phi_2(\tau \wedge\theta) v=0~. \label{e:step2-1-cdt-3} 
\end{gather}
Then $(\theta, \omega)\in \p E$.
\end{lemma}

The geometric idea underlying Lemma~\ref{lem:step2-1} is that \eqref{e:step2-1-cdt-1}, \eqref{e:step2-1-cdt-2} and \eqref{e:step2-1-cdt-3} ensure that there is a horizontal line containing $(\theta, \om)$ and meeting $\cup_{x\in \Nonchar(\partial E)} S \cap (x\cdot \Sigma) \subset \partial E$ in two distinct points.

 \begin{proof} Let $(\theta, \omega)\in S\setminus\Sigma$. Let $\xi, \tau \in \Anh(\omega)$ and $p,q \in\R$ be such that $\omega = \xi \wedge \tau$ and~\eqref{e:step2-1-cdt-1} holds true. Let $\epsilon>0$ be fixed small enough so that $(0,s \xi \wedge \theta + t \tau \wedge \theta) \in \calU$ for all $s,t\in (-\epsilon,\epsilon)$ and let $r\in \R\setminus\{1\}$, $u,v\in (-\epsilon,\epsilon)$ be such that~\eqref{e:step2-1-cdt-2} and~\eqref{e:step2-1-cdt-3} hold true. Since $\theta \not\in \Anh(\omega)$, we have  $p\xi+q \tau - \theta  \in \Lone \setminus \{0\}$ and we consider the horizontal line $\gamma(\R)$ where
 \begin{equation*}
 \gamma(t) := (\theta,\omega) \cdot (t(p\xi+q\tau - \theta) , 0)
 \end{equation*}
for $t\in \R$. We will verify that $\gamma(1)\in \partial E$ and $\gamma(r) \in \partial E$. This will imply by Lemma~\ref{linee} that $\gamma(\R) \subset \partial E$ and therefore $\gamma(0) = (\theta,\omega) \in \partial E$ as wanted.

We have $\gamma(1) = (p\xi+q\tau,\omega+p\theta\wedge\xi+q\theta\wedge\tau) \in \Lie(\xi+q\theta,\tau- p\theta)$ and hence $\gamma(1) \in \Sigma$. Since $\phi(\theta',\omega') = \phi_1(\theta') + \phi_2(\omega')$ for all $(\theta',\omega') \in \Sigma$, it follows from~\eqref{e:step2-1-cdt-1} that
\begin{equation*}
\begin{split}
\phi(\gamma(1)) &= \phi_1(p\xi+q\tau) + \phi_2(\omega+p\theta\wedge\xi+q\theta\wedge\tau)\\
&= p(\phi_1(\xi) + \phi_2(\theta\wedge\xi)) + q(\phi_1(\tau) + \phi_2(\theta\wedge\tau)) + \phi_2(\omega)\\
&= p \phi(\xi,\theta \wedge \xi) +  q \phi(\tau,\theta \wedge \tau) + \phi_2(\omega)=0~,
\end{split}
\end{equation*}
i.e., $\gamma(1)  \in S$. Therefore $\gamma(1) \in S\cap \Sigma$ and it follows from~\eqref{e:S-inter-Sigma-in-the boundary} that $\gamma(1) \in \partial E$.

To prove that $\gamma(r) \in \partial E$, we set $x:=(0,u \xi \wedge \theta + v \tau \wedge \theta)$ and we first verify that 
\begin{equation} \label{e:gamma(r)}
\gamma(r)  \in S \cap  (x \cdot \Sigma)~.
\end{equation} 
Since $\gamma(0) = (\theta,\omega) \in S$, $\gamma(1) \in S$, and since $S$ is the boundary of a precisely monotone subset of $\free$, we get from Lemma~\ref{linee} that $\gamma(r)  \in S$. We have $x^{-1} \cdot \gamma(r) = (\overline{\theta},\overline{\omega})$ where
\begin{equation*}
\overline{\theta} = (1-r)\theta + r(p\xi+q\tau)~, \quad \overline{\omega} = \omega + (rp+u) \theta \wedge \xi + (rq+v) \theta \wedge \tau~,
\end{equation*}
and $\overline{\theta} \wedge \overline{\omega} = (1-r -rpv  +rqu)\, \theta \wedge \omega$. Therefore it follows from~\eqref{e:step2-1-cdt-2} that $\overline{\theta} \wedge \overline{\omega}=0$, i.e., $(\overline{\theta},\overline{\omega}) \in \Sigma$ (see~\eqref{e:equation-Sigma}). Therefore $\gamma(r)  \in x \cdot \Sigma$, which concludes the proof of~\eqref{e:gamma(r)}.

We next verify that 
\begin{equation} \label{e:a-non-characteristic-point}
x \in  \Nonchar(\partial E) \cap \calU~.
\end{equation}
By~\eqref{e:step2-1-cdt-2} we have $\phi(x) = \phi(0,u \xi \wedge \theta + v \tau \wedge \theta) = u \phi_2(\xi \wedge \theta) + v \phi_2(\tau \wedge \theta)=0$, i.e., $x\in S$. Then~\eqref{e:a-non-characteristic-point} follows from  our choice of $\epsilon$ together with~\eqref{e:description} and~\eqref{e:no-characteristic-points}. Using Lemma~\ref{lem:step1}, we get that $S \cap (x \cdot \Sigma) \subset \partial E$ and hence~\eqref{e:gamma(r)} implies that $\gamma(r) \in \partial E$, which concludes the proof of the lemma.
\end{proof}

\begin{lemma} \label{lem:step2-2}
We set $F_1:=\{(\theta,\omega) \in  S \setminus \Sigma:\, \phi(\xi,\theta\wedge\xi) =0 \,\text{ for all } \xi \in \Anh (\omega)\}$. We also set $F_2=\{(\theta,\omega) \in S \setminus \Sigma:\, \phi_2(\omega) = 0\}$ if $\phi_2 \not= 0$ and $F_2 = \emptyset$ otherwise. Then $\Int_{S \setminus \Sigma} (F_1\cup F_2) = \emptyset$.
\end{lemma}

\begin{proof}
We will prove that $\Int_{S \setminus \Sigma} (F_1) =  \Int_{S \setminus \Sigma} (F_2)  =\emptyset$. Since $F_2$ is a relatively closed subset of $S\setminus \Sigma$, this will imply $\Int_{S \setminus \Sigma} (F_1\cup F_2) = \emptyset$, as wanted. For $j=1,2$ we denote by $\psi_j : \free \rightarrow \R$ the h-affine functions given by
$\psi_1(\theta,\omega) := \phi (\theta,0) =\phi_1(\theta)$ and  $\psi_2(\theta,\omega) := \phi (0,\omega) = \phi_2(\omega)$. Note for further use that $\psi_2 \not= \phi$.

\smallskip

To prove that $\Int_{S \setminus \Sigma} (F_1) = \emptyset$, we distinguish two cases.

\noindent \step{Case~1.} If $\phi=\psi_1$ then $F_1 \subset (\Ker \phi_1 \times (\Ker\phi_1 \wedge \Ker\phi_1) ) \setminus \Sigma$. Since $\phi_1\not=0$, we know that $\Ker\phi_1$ is a 2-dimensional linear subspace of $\Lone$. Therefore there are $\xi,\tau \in \Lone$ such that $\xi \wedge \tau \not=0$ and  $\Ker \phi_1 \times (\Ker\phi_1 \wedge \Ker\phi_1)= \Lie (\xi,\tau)\subset\Sigma$.  This implies that $F_1 =\emptyset$ and therefore $\Int_{S \setminus \Sigma} (F_1) = \emptyset$.

\noindent \step{Case~2.} If $\phi\not= \psi_1$, we consider $(\overline{\theta},\overline{\omega}) \in F_1$, an open neighborhood $\calO$ of $(\overline{\theta},\overline{\omega})$, and we shall verify that $\calO \cap S\setminus (\Sigma \cup F_1) \not=\emptyset$. Since $\Sigma$ is closed, we can assume with no loss of generality that $\calO \cap \Sigma =\emptyset$. Then we claim that there is $(\hat{\theta},\hat{\omega}) \in S \cap \calO$ such that $\psi_1(\hat{\theta},\hat{\omega}) \not=0$. Indeed otherwise $\{(\theta,\omega) \in \calO:\, \phi(\theta,\omega) = 0\} \subset \{(\theta,\omega) \in \calO:\, \psi_1(\theta,\omega) = 0\}$ and Corollary~\ref{cor:level-sets} implies $\phi= \psi_1$, which gives a contradiction. If $(\hat{\theta},\hat{\omega}) \not\in F_1$, we are done. If  $(\hat{\theta},\hat{\omega}) \in F_1$, let us consider the linear form $L_2: \omega \in \Ltwo \mapsto \phi(\hat{\theta},\hat{\omega} + \omega)$. Let $\hat{\xi}, \hat{\tau} \in \Anh (\hat{\omega})$ be such that $\hat{\omega} = \hat{\xi} \wedge \hat{\tau}$. Since $\hat{\theta} \not \in \Anh (\hat{\omega})$, we have $\dim (\Span\{\hat{\theta} \wedge \hat{\xi},\hat{\theta} \wedge \hat{\tau}\})=2$ and therefore $\Ker L_2 \cap \Span\{\hat{\theta} \wedge \hat{\xi},\hat{\theta} \wedge \hat{\tau}\} \not=\{0\}$. In other words there is $(u,v) \in \R^2 \setminus \{(0,0)\}$ such that $L_2(u \hat{\theta} \wedge \hat{\xi} + v \hat{\theta} \wedge \hat{\tau})=0$. For all $s\in \R$ we have $$\hat{\omega} + s(u \hat{\theta} \wedge \hat{\xi} + v \hat{\theta} \wedge \hat{\tau}) = (\hat{\xi} + sv\hat{\theta}) \wedge (\hat{\tau} - su \hat{\theta})$$ and therefore $\phi(\hat{\theta},(\hat{\xi} + su\hat{\theta}) \wedge (\hat{\tau} - sv \hat{\theta})) = sL_2(u \hat{\theta} \wedge \hat{\xi} + v \hat{\theta} \wedge \hat{\tau})=0$, i.e., $$(\hat{\theta},(\hat{\xi} + su\hat{\theta})\wedge (\hat{\tau} - sv \hat{\theta})) \in S~.$$
Let us now consider the linear form $L_1:\theta \in \Lone \mapsto \phi(\theta,\hat{\theta} \wedge \theta)$. By definition of $F_1$ we have $L_1(\xi) = 0$ for all $\xi \in \Anh (\hat{\omega})$. Since $L_1(\hat{\theta}) = \psi_1(\hat{\theta},\hat{\omega}) \not=0$, it follows that $L_1(\xi + s \hat{\theta}) \not=0$  for all $\xi \in \Anh (\hat{\omega})$ and all $s\in \R\setminus \{0\}$. Since $(u,v)\not=(0,0)$, we get that for all $s\in \R\setminus \{0\}$ either $L_1 (\hat{\xi} + su\hat{\theta}) \not=0$ or $L_1 (\hat{\tau} - sv \hat{\theta})\not=0$, i.e.,
$$(\hat{\theta},(\hat{\xi} + su\hat{\theta})\wedge (\hat{\tau} - sv \hat{\theta})) \not\in F_1~.$$
Then choosing $s\not=0$ small enough, we get that $(\hat{\theta},(\hat{\xi} + su\hat{\theta})\wedge (\hat{\tau} - sv \hat{\theta})) \in \calO \cap S \setminus F_1$, which proves that $\Int_{S \setminus \Sigma} (F_1) = \emptyset$, as wanted.

\smallskip

To prove that $\Int_{S \setminus \Sigma} (F_2) = \emptyset$ we only need to consider the case where $\phi_2\not=0$, otherwise the claim is obvious. If $\phi_2\not=0$,   we consider $(\overline{\theta},\overline{\omega}) \in F_2$ and an open neighborhood $\calO$ of $(\overline{\theta},\overline{\omega})$. Since $\Sigma$ is closed, we can assume with no loss of generality that $\calO \cap \Sigma =\emptyset$. Then we claim that there is $(\hat{\theta},\hat{\omega}) \in S \cap \calO$ such that $\phi_2(\hat{\omega}) = \psi_2(\hat{\theta},\hat{\omega}) \not=0$. Indeed otherwise $\{(\theta,\omega) \in \calO:\, \phi(\theta,\omega) = 0\} \subset \{(\theta,\omega) \in \calO:\, \psi_2(\theta,\omega) = 0\}$ and Corollary~\ref{cor:level-sets} implies $\phi= \psi_2$, which gives a contradiction. This shows that $\Int_{S \setminus \Sigma} (F_2) = \emptyset$ and concludes the proof of the lemma.
\end{proof}

\begin{lemma} \label{lem:step2}
We have $S\setminus \Sigma \subset \partial E$.
\end{lemma}

\begin{proof}
By Lemma~\ref{lem:step2-2} and since $\partial E$ is closed, we only need to prove that $S\setminus (\Sigma\cup F_1 \cup F_2) \subset \partial E$. We thus consider $(\theta,\omega) \in S\setminus (\Sigma\cup F_1 \cup F_2)$ and we claim that Lemma~\ref{lem:step2-1} can be applied to $(\theta,\omega)$, and therefore $(\theta,\omega)\in \partial E$, as wanted. To prove this claim, we let $\xi\in\Anh(\omega)$ be such that $\phi(\xi,\theta\wedge\xi)\not=0$ and $\tau \in \Anh(\omega)$ be such that $\omega = \xi \wedge \tau$. Since $\phi(\xi,\theta\wedge\xi)\not=0$, the set $V:=\{(p,q) \in \R^2: \eqref{e:step2-1-cdt-1} \text{ holds true}\}$ is a 1-dimensional affine subspace of $\R^2$. We then distinguish two cases. 

\smallskip

\noindent \step{Case~1.} If $\phi_2(\xi\wedge\theta) = \phi_2(\tau\wedge\theta) =0$, we let $(p,q) \in V\setminus\{(0,0)\}$. Then, given $\epsilon>0$, one can choose $r\in\R\setminus\{1\}$ close enough to 1 so that there are $u,v \in (-\epsilon,\epsilon)$ such that~\eqref{e:step2-1-cdt-2} holds true. Since $\phi_2(\xi\wedge\theta) = \phi_2(\tau\wedge\theta) =0$ ~\eqref{e:step2-1-cdt-3} holds true trivially and this concludes the proof in this first case.

\smallskip 

\noindent \step{Case~2.} If $(\phi_2(\xi\wedge\theta),\phi_2(\tau\wedge\theta)) \not=(0,0)$ then $\phi_2\not=0$. Since $(\theta,\omega) \not\in F_2$, we have $\phi_2(\omega)\not=0$ and therefore the 1-dimensional affine subspace $V$ of $\R^2$ does not contain the origin. On the other side, the 1-dimensional affine subspace $W:=\{(p,q) \in \R^2:\, \phi_2(\xi\wedge\theta) p + \phi_2(\tau\wedge\theta) q =0 \}$ of $\R^2$ contains the origin. Therefore $V\setminus W \not=\emptyset$ and we let $(p,q) \in V\setminus W$. We set $\delta:=\phi_2(\xi\wedge\theta) p + \phi_2(\tau\wedge\theta) q$. Then, given $r\in \R \setminus \{0,1\}$, there is a unique solution $(u,v) \in \R^2$ to~\eqref{e:step2-1-cdt-2} and~\eqref{e:step2-1-cdt-3} given by 
\begin{equation*}
u= - (1-r)(\delta r)^{-1}\phi_2(\tau\wedge\theta) \quad \text{and} \quad v=(1-r)(\delta r)^{-1}\phi_2(\xi\wedge\theta)~.
\end{equation*}
It follows that given $\epsilon>0$, one can choose $r\in\R\setminus\{1\}$ close enough to 1 so that the solution $(u,v) \in \R^2$ to~\eqref{e:step2-1-cdt-2} and~\eqref{e:step2-1-cdt-3} belongs to $(-\epsilon,\epsilon)^2$, which concludes the proof of the lemma.
\end{proof}

Putting together~\eqref{e:S-inter-Sigma-in-the boundary} and Lemma~\ref{lem:step2}  we get $S \subset\p E$.  Using left-translations, we also get   the following corollary. 

\begin{corollary} \label{cor:step2} Let $x\in \Nonchar(\partial E)$. Then there is a non constant h-affine function $\phi_x:\free \rightarrow \R$ such that  $x\in \Nonchar(S_x)$ and $S_x \subset \partial E$ where $S_x:=\{y\in\free:\, \phi_x(y) =0\}$. 
\end{corollary}

\begin{lemma} \label{lem:step3} We have $\partial E \subset S$.
\end{lemma}

\begin{proof}
We know from~\eqref{e:S-inter-Sigma-in-the boundary} and Lemma~\ref{lem:step2} that $S\subset \partial E$. We also know from Proposition~\ref{prop:noncharacteristic-points-exist} that $\Nonchar(\partial E)$ is a relatively dense subset of $\partial E$. Since $S$ is closed, it is therefore sufficient to prove that $\Nonchar(\partial E) \subset S$. We argue by contradiction and assume that $\Nonchar(\partial E) \setminus S \not=\emptyset$. We fix $\nu \in \Lambda^3\R^3 \setminus \{0\}$ and we let $\eta_0 \in \Lambda^0\R^3$, $\eta_1\in\Lone$, and $\eta_2 \in \Ltwo \setminus \{0\}$ be such that $\phi$ is given by~\eqref{e:phi}. 

\smallskip

We first claim that one can find $x=(\theta,\om) \in \Nonchar(\partial E) \setminus S$ in such a way that $\eta_1+\eta_0\theta \not=0$ whenever $\eta_0\not=0$. Indeed, if $x' \in  \Nonchar(\partial E) \setminus S$, we know from Proposition~\ref{foglietto} that $\Hor_{x'} \cap \partial E$ is a codimension-1 affine subspace of $\Hor_{x'}$. We also know that $\Nonchar(\partial E) \setminus S$ is an open subset of $\partial E$. Therefore, if $\eta_0\not=0$, one can find $x=(\theta,\om)$ close enough to $x'$ such that $x \in \Nonchar(\partial E) \setminus S$ and $\eta_1+\eta_0\theta \not=0$, as wanted. We then let $\widetilde{\phi}:\free \rightarrow\R$ be given by Corollary~\ref{cor:step2} applied to $x$ so that $\widetilde{\phi}$ is a non constant h-affine function, $\widetilde{S}:=\{y\in\free:\, \widetilde{\phi}(y) =0\} \subset \partial E$, and $x \in \Nonchar(\widetilde{S})\setminus S$.

\smallskip

We next claim that one can find $ \widetilde{x} \in \Nonchar(\widetilde{S}) \setminus S$ in such a way that $\Hor_{\widetilde{x}} \cap S \not= \emptyset$. Indeed, since $\widetilde{S}$ is the boundary of a precisely monotone subset of $\free$, we know that there is a 2-dimensional linear subspace $V$ of $\Hor_0$ such that $x \cdot V \subset \widetilde{S}$. We then let $\Omega$ denote an open neighborhood of the origin in $V$ such that $x \cdot \Omega \subset \Nonchar(\widetilde{S}) \setminus S$. Given  $\xi, \tau \in \Lone$, we have $(\theta,\om)\cdot (\xi,0)\cdot(\tau,0)\in S$ if and only if 
\begin{equation*}
\begin{aligned}
 \phi\left((\theta,\om)\cdot (\xi,0)\cdot(\tau,0)\right) \nu = \big\{ & 
 \eta_2\wedge\theta+\eta_1\wedge\om+\eta_0\theta\wedge\om 
 + [\eta_2+\eta_1\wedge\theta+\eta_0\om]\wedge\xi
 \big\}
 \\&
    +\big\{
    \eta_2+\eta_1\wedge\theta+\eta_0\om +[\eta_1+\eta_0\theta]\wedge\xi 
    \big\}_{(*)}    \wedge\tau=0
\end{aligned}
\end{equation*} 
To prove the claim, we shall now verify that one can find $\xi\in \Omega$  such that $\{\cdots\}_{(*)} \not=0$ in $\Lambda^2\R^3$. Indeed, if $\eta_0=0$ then $\{\cdots\}_{(*)}=\eta_2+\eta_1\wedge\theta+\eta_1\wedge\xi$. Since $\eta_2\neq 0$ and $\dim V = 2$, given any $\eta_1\in\Lone$, one can find $\xi\in \Omega$  such that $\{\cdots\}_{(*)} \neq 0$. If instead  $\eta_0\neq 0 $, by choice of $x=(\theta,\omega)$, we have $\eta_1+\eta_0\theta\neq 0$. Therefore, using once again the fact that $\dim V = 2$, one can also find in such a case $\xi\in \Omega$ such that $\{\cdots\}_{(*)} \not=0$. Next, for such a choice of $\xi$, the function $\tau \in \Lone \mapsto \phi\left((\theta,\om)\cdot (\xi,0)\cdot(\tau,0)\right)$ is surjective. Therefore one can find $\tau \in \Lone$ such that $\phi\left((\theta,\om)\cdot (\xi,0)\cdot(\tau,0)\right)=0$. Then, setting $\widetilde{x}:=(\theta,\omega) \cdot (\xi,0)$, we have $\widetilde{x}\in  \Nonchar(\widetilde{S}) \setminus S$ and $\widetilde{x} \cdot (\tau,0) \in \Hor_{\widetilde{x}} \cap S$ which concludes the proof of the claim.

\smallskip

We now claim that $\Hor_{\widetilde{x}} \cap \Nonchar(S) \setminus \widetilde{S} \not= \emptyset$. First, note that $\Hor_{\widetilde{x}} \cap \Char(S) = \emptyset$. Indeed, since $y\in \Hor_{\widetilde{x}}$ if and only if $\widetilde{x} \in \Hor_{y}$, if there is $y\in \Hor_{\widetilde{x}} \cap\Char(S)$ then $\widetilde{x} \in \Hor_{y} \subset S$, which gives a contradiction. Next, assume there is $\tau\in\Lone$ such that $\widetilde{x}\cdot(\tau,0) \in \Nonchar(S) \cap \widetilde{S}$. Since $\widetilde{x} \in \Hor_{\widetilde{x}\cdot(\tau,0)} \setminus S$, the horizontal line $\ell_\tau :=\{\widetilde{x} \cdot (t\tau,0):\, t\in \R\}$ intersects the smooth 5-dimensional submanifold $\Nonchar(S)$ transversally at $\widetilde{x}\cdot(\tau,0)$, see Lemma~\ref{lem:level-sets-as-submanifolds}. It follows that for all $\tau'\in \Lone$ close enough to $\tau$, the horizontal line $\ell_{\tau'}:=\{\widetilde{x} \cdot (t\tau',0):\, t\in \R\}$ intersects $\Nonchar(S)$ transversally at some point close to $\widetilde{x}\cdot(\tau,0)$ and hence $\ell_{\tau'} \cap \Nonchar(S) \not= \emptyset$. Since $\widetilde{x} \in \Nonchar(\widetilde{S})$, one can moreover choose such a $\tau'$ so that $\ell_{\tau'} \cap \widetilde{S} = \{\widetilde{x}\}$. Since $\widetilde{x} \not\in S$, for such a choice of $\tau'$, we then have $\ell_{\tau'} \cap \Nonchar(S) \cap \widetilde{S} =\emptyset$. All together it follows that for such a choice of $\tau'$, we have $\emptyset \not= \ell_{\tau'} \cap \Nonchar(S) \setminus \widetilde{S} \subset \Hor_{\widetilde{x}}$ which concludes the proof of the claim.

\smallskip

We thus have proved that there are $\widetilde{x} \in \Nonchar(\widetilde{S}) \setminus S$ and $\tau \in \Lone$ such that $\widetilde{x}\cdot (\tau,0) \in \Nonchar(S) \setminus \widetilde{S}$. Then it follows from Lemma~\ref{lem:level-sets-as-submanifolds} that the horizontal line  $\ell_\tau :=\{\widetilde{x} \cdot (t\tau,0):\, t\in \R\}$ intersects transversally the smooth 5-dimensional submanifolds $\Nonchar(\widetilde{S})$ and $\Nonchar(S)$ at respectively $\widetilde{x}$ and $\widetilde{x}\cdot (\tau,0)$. Therefore there is an open neighborhood $\calO_{\widetilde{x}}$ of $\widetilde{x}$ such that for every $y\in \widetilde{S}_{\widetilde{x}}:= \calO_{\widetilde{x}} \cap \Nonchar(\widetilde{S})$ the following hold true. First, the horizontal line $\ell^y:=\{y\cdot (t\tau,0):\, t\in \R\}$ intersects $\Nonchar(\widetilde{S})$ transversally at $y$ and hence the set $T:=\{y\cdot (t\tau,0):\, y \in \widetilde{S}_{\widetilde{x}},\, t\in \R\}$ has non empty interior. Second, the horizontal line $\ell^y$  intersects $\Nonchar(S)$ transversally at some point close to $\widetilde{x}\cdot (\tau,0)$. Since $\widetilde{S} \cup S \subset \partial E$, it then follows from Lemma~\ref{linee} that $T\subset \partial E$. This implies in turn that $\Int(\partial E) \not= \emptyset$ which contradicts Proposition~\ref{prop:noncharacteristic-points-exist} and concludes the proof of the lemma.
\end{proof}

\begin{lemma} \label{lem:step4} We have $\Int(E) = \{y\in \free :\, \phi(y) < 0\}$ and $\overline{E}= \{y\in \free :\, \phi(y) \leq 0\}$.
\end{lemma}

\begin{proof}
We know from~\eqref{e:S-inter-Sigma-in-the boundary}, Lemma~\ref{lem:step2} and Lemma~\ref{lem:step3} that $S=\p E$.  By~\eqref{e:description} we have $\Int(E) \cap C^- \not=\emptyset$ and $\Int(E^c) \cap C^+ \not=\emptyset$ where $C^- := \{y\in \free :\, \phi(y) < 0\} $ and $ C^+ := \{y\in \free :\, \phi(y) > 0\}~$. We also know from Proposition~\ref{prop:connectedness-h-affine-maps} that $C^-$ and $C^+$ are the connected components of $S^c=(\partial E)^c = \Int(E) \cup \Int(E^c)$. Then the conclusion follows from elementary topological considerations.
\end{proof}


\section{Sublevel and level sets of h-affine functions on \texorpdfstring{$\free$}{free}} \label{sect:(sub)levelsets-h-affine-maps}

We prove in this section properties of sublevel and level sets of h-affine functions on $\free$ that have been used in the Sections~\ref{sect:local-statement-free-case} and~\ref{sect:global-statement-free-case} and may have their own interest. Throughout this section we identify $\free$ with $\Lone \times \Ltwo$, and $\Lambda^3 \R^3$ with $\R$ via $s\in \R \mapsto s\nu \in \Lambda^3\R^3$ where $\nu \in \Lambda^3\R^3 \setminus \{0\}$ is fixed. With these identifications, we recall that we are interested in functions $\phi:\free \rightarrow \R$ such that there is $(\eta_0, \eta_1,\eta_2,\eta_3)\in \Lambda^0\R^3\times\Lambda^1\R^3\times\Lambda^2\R^3 \times \Lambda^3 \R^3$ such that
\begin{equation} \label{e:phi-for-sect5}
\phi(\theta,\omega)  = \eta_3 +  \eta_2 \wedge \theta + \eta_1 \wedge \omega + \eta_0 \, \theta \wedge \omega
\end{equation}
for all $(\theta,\omega) \in \free$. As already mentioned in Section~\ref{sect:PM-step2-Carnot-algebras}, such functions can easily be seen to be h-affine. Let us recall for the sake of completeness that it has been proved in~\cite[Theorem~1.1]{LeDonneMorbidelliRigot1} that all h-affine functions on $\free$ are of this form.  We will not need this nontrivial result here, except for the use of the terminology "h-affine function" that will denote a function $\phi:\free \rightarrow \R$ of the form~\eqref{e:phi-for-sect5} throughout this section.

\begin{proposition} \label{prop:connectedness-h-affine-maps}
Let $\phi:\free \rightarrow \R$ be a non constant h-affine function. Then $\phi$ is surjective and for every $c\in \R$ the sets $\{x\in \free :\, \phi(x) < c\}$ and $\{x\in \free :\, \phi(x) > c\}$ are the connected components of $\{x\in \free :\, \phi(x) \not= c\}$. 
\end{proposition}

\begin{proof}
Let $(\eta_0, \eta_1,\eta_2,\eta_3)\in \Lambda^0\R^3\times\Lambda^1\R^3\times\Lambda^2\R^3 \times \Lambda^3 \R^3$ with $(\eta_0,\eta_1,\eta_2) \not= (0,0,0)$ be such that $\phi$ is given by~\eqref{e:phi-for-sect5}. If $\eta_0 = 0$ then $\phi$ is a non constant affine function on $\free$  seen as a vector space and the statement is obvious. We thus assume that $\eta_0 \not=0$ in the rest of this proof. For $(\theta,\omega) \in \free$, we have
$$(\eta_0 \, \theta + \eta_1 ) \wedge (\omega + \eta_0^{-1} \, \eta_2) =  \eta_2 \wedge \theta + \eta_1 \wedge \omega  + \eta_0\, \theta \wedge \omega + \eta_0^{-1} \, \eta_1 \wedge \eta_2~.$$
Since the map $(\theta,\omega) \in \free \mapsto (\eta_0 \, \theta + \eta_1, \omega + \eta_0^{-1} \, \eta_2) \in \free$ is a homeomorphism, we thus only need to consider the case where $\phi(\theta,\omega) = \theta \wedge \omega$. In such a case $\phi$ is a quadratic form with signature $(0,3,3)$ on $\free$ seen as a vector space and hence is in particular surjective. Let $c\in \R$ be given. Since $\phi$ is surjective, the sets $\{x\in \free :\, \phi(x) < c\}$ and $\{x\in \free :\, \phi(x) > c\}$ are non empty. Let us verify that $\{(\theta,\omega) \in \free:\, \phi(x) < c\}$  is arcwise connected, the proof for the set $\{x\in \free :\, \phi(x) > c\}$ being similar. Since $\phi$ is a quadratic form with signature $(0,3,3)$, this is equivalent to proving that $F:=\{(u, v)\in\R^3\times\R^3 :\, \|v\|^2 - \|u\|^2<c\}$ is arcwise connected where $\|\cdot\|$ denotes a Euclidean norm on $\R^3$. It can easily be seen that any two points $(\ol u,\ol v)$, $(\wh u,\wh v)\in F$ can be connected by a concatenation of three continuous paths contained in $F$, namely, the segment from $(\ol u, \ol v)$ to $(\ol u, 0)$, any continuous path connecting $(\ol u, 0)$ and $(\wh u, 0)$ inside the arcwise connected subset $\{(u, 0)\in\R^3\times\R^3 :\, - \|u\|^2<c\}$ of $F$, and the segment from $(\wh u, 0)$ to $(\wh u, \wh v)$. To conclude the proof of the proposition, note that it follows from the surjectivity and continuity of $\phi$ that the set $\{x\in \free :\, \phi(x) \not= c\}$ is not connected. Therefore the sets $\{x\in \free :\, \phi(x) < c\}$ and $\{x\in \free :\, \phi(x) > c\}$ are the two connected components of $\{x\in \free :\, \phi(x) \not= c\}$, as claimed.
\end{proof}

\begin{lemma} \label{lem:connected-sublevelsets-local} Let $\phi:\free \rightarrow \R$ be a non constant h-affine function, $x\in \free$, and $\calO$ be an open neighborhhood of $x$. Then there is an open neighborhood $\calO' \subset \calO$ of $x$ such that the non empty sets $\{y\in \calO' :\, \phi(y) < \phi(x)\}$ and $\{y\in \calO' :\, \phi(y) > \phi(x)\}$ are connected.
\end{lemma}

\begin{proof}
Arguing as in the proof of Proposition~\ref{prop:connectedness-h-affine-maps} we only need to consider the case where $\phi(\theta,\omega) = \theta \wedge \omega$. If $\phi(x) \not=0$, the claim follows from the fact that the set $\{(\theta,\omega)\in \free :\, \theta \wedge \omega = \phi(x)\}$ is a smooth 5-dimensional submanifold of $\free$. Similarly, the set $\{(\theta,\omega)\in \free :\, \theta \wedge \omega = 0\} \setminus \{(0,0)\}$ is a smooth 5-dimensional submanifold of $\free$ and we thus only need to consider the case where $x=0$. Arguing as in the proof of Proposition~\ref{prop:connectedness-h-affine-maps}, we are lead to show that the set  $\{(u,v)\in B(0,\e)\times B(0,\e): \|v\|<\|u\|\}$ is arcwise connected for all $\e>0$ where $B(0,\e)$ denotes a Euclidean open ball in $\R^3$ centered at the origin with radius $\e$. This can be done in the same way than in the proof of Proposition~\ref{prop:connectedness-h-affine-maps} taking care that the intermediate path from $(\ol u, 0)$ to $(\wh u,0)$ remains contained in $(B(0,\e) \setminus\{0\}) \times \{0\}$.
\end{proof}

\begin{proposition} \label{prop:comparison-of-sublevelsets-local}
Let $\phi,\psi:\free \rightarrow \R$ be h-affine functions with $\psi$ non constant. Assume that there is $x \in \free$ such that $\phi(x) = \psi(x)$ and there is an open neighborhood $\calO$ of $x$ such that 
\begin{equation*} 
\{y \in \calO:\, \phi(y) \leq \phi(x)\} \subset \{ y \in\calO:\, \psi(y)  \leq  \psi(x)\}~.
\end{equation*}
Then there is $\lambda > 0$ such that $\psi = \lambda \phi$.
\end{proposition}

\begin{proof}
Considering the h-affine functions $y \mapsto \phi(x\cdot y) - \phi(x)$ and $y \mapsto \psi(x \cdot y) - \psi(x)$, we can assume with no loss of generality that $x=0$ and $\phi(0)=\psi(0)=0$.  We set
$$E_\phi:=\{y\in\free:\, \phi(y) \leq 0\} \quad \text{and} \quad E_\psi:=\{y\in\free:\, \psi(y) \leq 0\}~.$$
By assumption there is an open neighborhood $\calO$ of $0$ such that 
\begin{equation}  \label{e:hyp-local-inclusion}
E_\phi \cap \calO \subset E_\psi \cap \calO~.
\end{equation}

We first verify that $\phi\not= 0$. To prove this claim, we argue by contradiction and assume that $\phi= 0$. Then $E_\phi \cap \calO = \calO$ and~\eqref{e:hyp-local-inclusion} implies that $\psi(y) \leq 0$ for all $y \in \calO$. Since $\psi \not=0$, there is $(\alpha_0, \alpha_1, \alpha_2) \in (\Lambda^0 \R^3 \times \Lone \times \Ltwo) \setminus \{(0,0,0)\}$ such that $$\psi(\theta,\omega) = \alpha_2\wedge\theta + \alpha_1 \wedge \omega + \alpha_0\, \theta \wedge \omega.$$  The fact that $\psi \leq 0$ in $\calO \cap (\Lone \times \{0\})$ implies that $\alpha_2 = 0$. Similarly, the fact that $\psi \leq 0$ in $\calO \cap (\{0\}\times \Ltwo)$ implies that $\alpha_1 = 0$. Finally it can easily be seen that if $ \alpha_0\, \theta \wedge \omega \leq 0$ for all $(\theta,\omega) \in \calO$ then $\alpha_0 =0$. Therefore $(\alpha_0, \alpha_1, \alpha_2) = (0,0,0)$ which gives a contradiction. 

\smallskip

Since $\phi\not= 0$, there is $(\eta_0, \eta_1, \eta_2)\in ( \Lambda^0 \R^3 \times \Lone \times \Ltwo ) \setminus \{(0,0,0)\}$ such that 
$$\phi(\theta,\omega) = \eta_2\wedge\theta + \eta_1 \wedge \omega + \eta_0\, \theta \wedge \omega.$$ 
Before we prove the proposition, we begin with some preliminary facts.
\begin{equation} \tag{FACT~1}\label{e:eta12-not-0-if-alpha12-not-0}
(\alpha_1 , \alpha_2)  \not= (0,0) \Rightarrow (\eta_1,\eta_2) \not= (0,0)~.
\end{equation}
By contradiction, assume that $(\alpha_1 , \alpha_2) \not= (0,0)$ and $(\eta_1,\eta_2) = (0,0)$. Then~\eqref{e:hyp-local-inclusion} reads as 
\begin{equation*}
\begin{split}
E_\phi\cap\calO &=\{(\theta,\om)\in\calO:\,  \eta_0\, \theta \wedge \omega \leq 0\}\\
& \subset \{(\theta,\om)\in\calO:\, \alpha_2\wedge\theta + \alpha_1 \wedge \omega + \alpha_0\, \theta \wedge \omega \leq 0\}~.
\end{split}
\end{equation*}
In particular there is a neighborhood  $U$ of the origin in $\Lone$ so that $( \theta, \theta\wedge\xi)\in E_\phi\cap\calO$ for all $\theta,\xi\in U$.  Therefore the previous inclusion implies that  $\alpha_2\wedge\theta + \alpha_1\wedge \theta\wedge\xi\leq 0 $ for all $\theta,\xi\in U$. Taking $\xi=0$, it follows that $\alpha_2\wedge\theta \leq 0$ for all $\theta\in U$ which implies  that $\alpha_2=0$. Then we get that  $\alpha_1\wedge \theta\wedge\xi\leq 0 $ for all $\theta,\xi\in U$ which implies in turn that $\alpha_1=0$. Therefore $(\alpha_1 , \alpha_2) = (0,0)$ which gives a contradiction and concludes the proof of~\ref{e:eta12-not-0-if-alpha12-not-0}.

\begin{equation*} \tag{FACT~2}\label{e:eta12-and-alpha12-colinear}
(\alpha_1 , \alpha_2)  \not= (0,0) \text{ and }  (\eta_1,\eta_2) \not= (0,0) \Rightarrow \text{there is } \lambda>0 \text{ such that } (\alpha_1 , \alpha_2) = \lambda (\eta_1,\eta_2)~.
\end{equation*}
Let $(\theta,\omega) \in \free$ be such that $\eta_2\wedge\theta + \eta_1 \wedge \omega <0$. For $t>0$ small enough, $(t\theta,t\omega) \in \calO$ and $\phi(t\theta, t\omega)    = t(\eta_2\wedge\theta + \eta_1 \wedge \omega) + t^2\, \eta_0 \,\theta \wedge \omega \leq 0$,~i.e., $(t\theta,t\omega) \in E_\phi \cap \calO$. Then~\eqref{e:hyp-local-inclusion} implies that $(t\theta,t\omega) \in E_\psi$,~i.e., $t(\alpha_2\wedge\theta + \alpha_1 \wedge \omega) + t^2\, \alpha_0\, \theta \wedge \omega \leq 0$ 
for all $t>0$ small enough. This implies in turn that $\alpha_2\wedge\theta + \alpha_1 \wedge \omega \leq 0$. Therefore 
$$\{(\theta,\om)\in\free:\, \eta_2\wedge\theta + \eta_1 \wedge \omega < 0\} \subset \{(\theta,\om)\in\free:\, \alpha_2\wedge\theta + \alpha_1 \wedge \omega  \leq 0\}~.$$
Since $(\eta_1,\eta_2) \not= (0,0)$ and $(\alpha_1,\alpha_2)\not=(0,0)$,  this is an inclusion between half-spaces in $\free$ seen as a vector space with the origin as a common point in their boundary. It implies in turn that there is $\lambda>0$ such that $(\alpha_1 , \alpha_2) = \lambda (\eta_1,\eta_2)$ and concludes the proof of~\ref{e:eta12-and-alpha12-colinear}.

\begin{equation*} \tag{FACT~3} \label{e:when-alpha0-is-0}
 \alpha_0=0 \Rightarrow \eta_0 = 0~.
\end{equation*} 
By contradiction, assume that $\alpha_0=0$, and hence $(\alpha_1,\alpha_2) \not= (0,0)$, and $\eta_0 \not= 0$. Since $(\alpha_1,\alpha_2) \not= (0,0)$, we know from~\ref{e:eta12-not-0-if-alpha12-not-0} and~\ref{e:eta12-and-alpha12-colinear} that there is $\lambda>0$ such that $(\alpha_1,\alpha_2) = \lambda (\eta_1,\eta_2)$. Then~\eqref{e:hyp-local-inclusion} reads as
\begin{equation*}
 \{(\theta,\omega)\in\O: \eta_2\wedge\theta+\eta_1\wedge\omega+\eta_0\theta\wedge\omega\leq 0\}
\subset  \{(\theta,\omega)\in\O: \eta_2\wedge\theta+\eta_1\wedge\omega \leq 0\}.
\end{equation*}
By~Lemma~\ref{lem:Sigma^+--meets-hyperplane-local}, to be proved below, there is $(\theta,\omega)\in\calO$ such that   
$\eta_2\wedge\theta+\eta_1\wedge\omega+\eta_0\theta\wedge\omega=0$ and $\eta_0\theta\wedge\omega<0$. Thus $\eta_2\wedge\theta+\eta_1\wedge\omega>0$ which contradicts the inclusion above and concludes the proof of~\ref{e:when-alpha0-is-0}.

\begin{equation*} \tag{FACT~4}\label{e:when-eta0-is-0}
(\alpha_1,\alpha_2) \not= (0,0) \text{ and } \eta_0=0 \Rightarrow \alpha_0 = 0~.
\end{equation*} 
By contradiction, assume that $(\alpha_1,\alpha_2) \not= (0,0)$, $\eta_0=0$, and $\alpha_0\neq 0$. On the one hand, using~\ref{e:eta12-not-0-if-alpha12-not-0} and~\ref{e:eta12-and-alpha12-colinear}, \eqref{e:hyp-local-inclusion} reads as
\begin{equation*}
 \{(\theta,\omega)\in\calO: \alpha_2 \wedge\theta+\alpha_1\wedge\omega \leq 0\}
 \subset \{(\theta,\omega)\in\calO: \alpha_2 \wedge\theta+\alpha_1\wedge\omega+\alpha_0\theta\wedge\omega\leq 0\}.
\end{equation*}
On the other hand, by Lemma~\ref{lem:Sigma^+--meets-hyperplane-local}, there is $(\theta,\omega)\in\calO$ such that $\alpha_2\wedge\theta+\alpha_1\wedge\omega=0$ and $\alpha_0\theta\wedge\omega>0$ which contradicts the inclusion above and concludes the proof of~\ref{e:when-eta0-is-0}.

\begin{equation*} \tag{FACT~5}\label{e:alpha12-are-0,0-implies-eta12-is-0}
(\alpha_1,\alpha_2) = (0,0) \Rightarrow (\eta_1,\eta_2) = (0,0)~.
\end{equation*}
By contradiction, assume that $(\alpha_1,\alpha_2) = (0,0)$ and $(\eta_1,\eta_2)\not = (0,0)$. Then~\eqref{e:hyp-local-inclusion} reads as
\begin{equation*}
 \{(\theta,\om)\in\calO: \eta_2\wedge\theta+\eta_1\wedge
 \omega+\eta_0\theta\wedge\omega\leq 0\}\subset
  \{(\theta,\om)\in\calO:  \alpha_0\theta\wedge\omega\leq 0\},
\end{equation*}
where $\alpha_0\neq 0$. By Lemma~\ref{lem:Sigma^+--meets-hyperplane-local}, there is $(\theta,\om)\in\calO $ such that $\eta_2\wedge\theta+\eta_1\wedge \omega+\eta_0\theta\wedge\omega=0$ and $\alpha_0\theta\wedge\omega>0$ which contradicts the inclusion above and concludes the proof of~\ref{e:alpha12-are-0,0-implies-eta12-is-0}.

\medskip
We now turn to the proof of Proposition~\ref{prop:comparison-of-sublevelsets-local}. We divide it into three cases.

\noindent \step{Case~1.} We first consider the case where $(\alpha_1,\alpha_2)\not=(0,0)$ and $\alpha_0=0$. Then we know from~\ref{e:eta12-not-0-if-alpha12-not-0} and~\ref{e:eta12-and-alpha12-colinear} that there is $\lambda>0$ such that $(\alpha_1,\alpha_2) = \lambda (\eta_1,\eta_2)$. We also know from~\ref{e:when-alpha0-is-0} that $\eta_0 =0$. Therefore $(\alpha_0,\alpha_1,\alpha_2) =  \lambda (\eta_0,\eta_1,\eta_2)$ as wanted.

\noindent \step{Case~2.} We next consider the case where $(\alpha_1,\alpha_2)\not=(0,0)$ and $\alpha_0\not=0$. Then we once again know from~\ref{e:eta12-not-0-if-alpha12-not-0} and~\ref{e:eta12-and-alpha12-colinear} that there is $\lambda>0$ such that $(\alpha_1,\alpha_2) = \lambda (\eta_1,\eta_2)$. We also know from~\ref{e:when-eta0-is-0} that $\eta_0\not=0$. Therefore there is $\mu\not=0$ such that $\alpha_0 = \mu \eta_0$. Setting $s:=\mu/\lambda$, the assumption~\eqref{e:hyp-local-inclusion} reads as
\begin{multline} \label{e:hyp-case2}
\{(\theta,\om)\in\calO:\, \eta_2\wedge\theta + \eta_1 \wedge \omega  + \eta_0 \, \theta \wedge \omega \leq 0\} \\ \subset \{(\theta,\om)\in \calO:\, \eta_2\wedge\theta + \eta_1 \wedge \omega + s \eta_0 \, \theta \wedge \omega \leq 0\}
\end{multline}
with $(\eta_1,\eta_2)\not=0$, $s\eta_0\not=0$, and, to conclude the proof of Proposition~\ref{prop:comparison-of-sublevelsets-local} in the present case, we shall verify that $s=1$. By Lemma~\ref{lem:Sigma^+--meets-hyperplane-local}, there is $(\theta,\omega)\in\calO$ such that $\eta_2\wedge\theta+\eta_1\wedge\omega+\eta_0\theta\wedge\omega=0$ and $ \theta\wedge\omega>0$. Then~\eqref{e:hyp-case2} implies that $(s-1)\eta_0 \leq0$. Similarly, once again by Lemma~\ref{lem:Sigma^+--meets-hyperplane-local}, there is $(\theta',\omega')$ such that $\eta_2\wedge\theta'+\eta_1\wedge \omega'+\eta_0\theta'\wedge\omega'=0$ and  $\theta'\wedge\omega'<0$, and~\eqref{e:hyp-case2} implies now the opposite inequality $(s-1)\eta_0\geq 0$. Therefore $s=1$ as wanted.

\noindent \step{Case~3.} We finally consider the case where $(\alpha_1,\alpha_2) = (0,0)$. Then we know from~\ref{e:alpha12-are-0,0-implies-eta12-is-0} that $(\eta_1,\eta_2) = (0,0)$ and~\eqref{e:hyp-local-inclusion} reads as
\begin{equation} \label{e:hyp-case3}
\{(\theta,\om)\in\calO:\, \eta_0 \, \theta \wedge \omega \leq 0\} \subset \{(\theta,\om)\in\calO:\, \alpha_0 \, \theta \wedge \omega \leq 0 \} 
\end{equation}
with $\alpha_0 \not=0$ and $\eta_0\not=0$. Considering $(\theta,\omega)\in\calO$ such that $\eta_0\theta\wedge\omega<0$ we get that $\eta_0\alpha_0 >0$, i.e., there is $\lambda>0$ such that $\alpha_0 =\lambda \eta_0$. Therefore $(\alpha_0,\alpha_1,\alpha_2)  = \lambda (\eta_0,\eta_1,\eta_2)$ and this concludes the proof of the proposition.
\end{proof}

\begin{lemma} \label{lem:Sigma^+--meets-hyperplane-local}
Set $\Sigma^+:= \{(\theta,\om)\in\free:\, \theta \wedge \omega > 0\}$ and $\Sigma^-:= \{(\theta,\om)\in\free:\, \theta \wedge \omega < 0\}$. Let $\calO \subset \free$ be an open neighborhood of the origin. Let $(\eta_0,\eta_1,\eta_2) \in \Lambda^0\R^3 \times \Lone \times \Ltwo $ with $(\eta_1,\eta_2) \not= (0,0)$. Then
\begin{align} 
\label{e:useful-1} & \Sigma^+\cap\{ (\theta,\om)\in\calO:\eta_2\wedge \theta + \eta_1 \wedge \omega+\eta_0
\theta\wedge\omega=0  \}\neq \varnothing~,
\\ &
\Sigma^-\cap \{ (\theta,\om)\in\calO:\eta_2\wedge \theta + \eta_1 \wedge \omega+\eta_0\theta\wedge\omega=0  \}
\neq \varnothing~.
\label{e:useful-2}
\end{align}
\end{lemma}

\begin{proof}
We first prove the lemma when $\eta_0=0$. Set $V:=\{(\theta,\omega)\in\free: \eta_2\wedge\theta+\eta_1\wedge\omega=0\}$. Then $V$ is a 5-dimensional linear subspace of $\free$. The quadratic form $(\theta,\om)\mapsto \theta\wedge\om $ on $\free$ seen as a vector space has signature $(0,3,3)$. Therefore there are 3-dimensional linear subspaces $W^+$ and $W^-$ of $\free$ such that $\theta \wedge \omega >0$ for all $(\theta,\omega) \in W^+$ and $\theta \wedge \omega < 0$ for all $(\theta,\omega) \in W^-$. Since $\dim V =5$, we have $W^+ \cap V \not=\emptyset$ and $W^- \cap V \not=\emptyset$ which proves~\eqref{e:useful-1} and~\eqref{e:useful-2} when $\eta_0 = 0$.

Assume that $\eta_0 \not=0$ and let us prove~\eqref{e:useful-1}, the proof of~\eqref{e:useful-2} being similar. Set $\phi(\theta,\omega):=\eta_2\wedge \theta+\eta_1\wedge\omega+\eta_0\theta\wedge\omega$. By the previous argument, there is $(\ol\theta,\ol\om)\in \calO$ such that $\eta_2\wedge\ol\theta+\eta_1\wedge\ol\omega=0$ and $ \ol\theta\wedge\ol\omega>0$. Assume that $\eta_0>0$. Then choose $(\wh\theta,\wh\omega)\in\calO$ close enough to $(\ol\theta,\ol\om)$ such that $-\eta_0\wh\theta\wedge\wh\omega <\eta_2\wedge\wh\theta+\eta_1\wedge \wh\omega<0$. On the one hand, we have $\phi(\wh\theta,\wh\omega)>0$. On the other hand, for $t>0$ small enough, we have $\phi(t\wh\theta, t\wh\omega)=t(\eta_2\wedge\wh\theta+\eta_1\wedge\wh\omega)+ t^2\eta_0 \wh\theta\wedge\wh\omega <0$. Therefore there is $\hat t\in(0,1)$ such that $\phi(\hat t\hat\theta, \hat t\hat\omega)=0$. Since $(\hat t\hat\theta, \hat t\hat\omega) \in \Sigma^+$ this concludes the proof of~\eqref{e:useful-1} when $\eta_0 >0$. If $\eta_0<0$, we choose $(\theta^*, \omega^*)\in\calO $ close enough to $(\ol\theta,\ol\om)$ such that $0<\eta_2\wedge\theta^*+\eta_1\wedge\omega^*<-\eta_0\theta^*\wedge\omega^*$ and we argue in a similar way.
\end{proof}

Combining Proposition~\ref{prop:comparison-of-sublevelsets-local} with topological arguments, we get the following corollary where sublevel sets are replaced by level sets.

\begin{corollary} \label{cor:level-sets}
Let $\phi,\psi:\free \rightarrow \R$ be h-affine functions with $\psi$ non constant. Assume that there is $x \in \free$ such that $\phi(x) = \psi(x)$ and there is an open neighborhood $\calO$ of $x$ such that 
\begin{equation} \label{e:hyp-levelsets}
\{y \in \calO:\, \phi(y) = \phi(x)\} \subset \{ y \in\calO:\, \psi(y)  = \psi(x)\}~.
\end{equation}
Then there is $\lambda \in \R \setminus \{0\}$ such that $\psi = \lambda \phi$.
\end{corollary}

\begin{proof}
Arguing as in the beginning of the proof of Proposition~\ref{prop:comparison-of-sublevelsets-local}, we can assume 
with no loss of generality that $x=0$ and $\phi(x) = \psi(x) = 0$. Since $\psi:\free \rightarrow \R$ is a non constant h-affine function, shrinking $\calO$ if necessary, one can assume that $\{ y \in\calO:\, \psi(y)  < 0\}$ is connected, see Lemma~\ref{lem:connected-sublevelsets-local}. One can also easily verify that $ \{ y \in\calO:\, \psi(y)  = 0\} \subset \overline{\{ y \in\calO:\, \psi(y)  < 0\}}$. Then it follows from Lemma~\ref{lem:sublevel-sets-topological-lemma}, to be proved below, that either $\{y \in \calO:\, \psi(y) \leq 0\} \subset \{ y \in\calO:\, \phi(y)  \leq 0\}$ or $\{y \in \calO:\, \psi(y) \leq 0\} \subset \{ y \in\calO:\, \phi(y)  \geq 0\}$. Arguing as in the beginning of the proof of Proposition~\ref{prop:comparison-of-sublevelsets-local}, one can also verify that the fact that $\psi$ is non constant together with~\eqref{e:hyp-levelsets} implies that $\phi$ is non constant as well. Then, changing $\phi$ into $-\phi$ if necessary, one can apply Proposition~\ref{prop:comparison-of-sublevelsets-local} to get the required conclusion. 
\end{proof}

Given a space $X$, a subset $\calO$ of $X$, and $\phi:X \rightarrow \R$, we set
\begin{align*}
\calO_\phi^0 & := \{y\in \calO:\, \phi(y)= 0\}~,
\\ 
\calO_\phi^- := \{y\in \calO:\, \phi(y)< 0 \} & \quad \text{and} \quad \calO_\phi^+ := \{y\in \calO:\, \phi(y)>0\}~.
\end{align*}

\begin{lemma} \label{lem:sublevel-sets-topological-lemma} Let $X$ be a topological space, $\calO\subset X$ be open, and $\phi, \psi:X \rightarrow \R$. Assume that $\phi$ is continuous. Assume also that  $\calO_\phi^0 \subset \calO_\psi^0\subset \overline{\calO_\psi^-}$ and $\calO_\psi^-$ is connected. Then either $\calO_\psi^- \cup \calO_\psi^0 \subset \calO_\phi^- \cup \calO_\phi^0$ or $\calO_\psi^- \cup \calO_\psi^0 \subset \calO_\phi^+ \cup \calO_\phi^0$.
\end{lemma}

\begin{proof}
Since $\calO_\psi^- \cap \calO_\phi^0 = \emptyset$, we have $\calO_\psi^- = (\calO_\psi^- \cap \calO_\phi^-) \sqcup (\calO_\psi^- \cap \calO_\phi^+)$. Since $\calO_\psi^-$ is connected and $\phi$ is continuous, it follows that either  
$$\calO_\psi^- = \calO_\psi^- \cap \calO_\phi^- \subset \calO_\phi^- \quad \text{or} \quad \calO_\psi^- = \calO_\psi^- \cap \calO_\phi^+ \subset \calO_\phi^+~.$$ 
Since $\calO_\psi^0\subset \overline{\calO_\psi^-} \cap \calO$, we also get that either 
$$\calO_\psi^0  \subset \overline{\calO_\phi^-} \cap \calO \subset \calO_\phi^- \cup \calO_\phi^0 \quad \text{or} \quad \calO_\psi^0 \subset  \overline{\calO_\phi^+} \cap \calO \subset \calO_\phi^+ \cup \calO_\phi^0$$
which concludes the proof of the lemma.
\end{proof}

We conclude this section with the proof of rather easy properties of the set of noncharacterictic points of level sets of h-affine functions that has been used in the proof of Lemma~\ref{lem:step3}.

\begin{lemma} \label{lem:level-sets-as-submanifolds}
Let $\phi:\free \rightarrow \R$ be a non constant h-affine function, $c\in\R$, and set $S:=\{x\in \free:\, \phi(x) = c\}$. Then $\Nonchar(S)$ is a smooth 5-dimensional submanifold of $\free$ and for all $x\in\Nonchar(S)$ we have $T_x S \cap \Hor_x = S \cap \Hor_x$ where $T_x S$ denotes the tangent space to $S$ at $x$ seen as an affine subspace of $\free$ through $x$.
\end{lemma}
 
\begin{proof}
Assume with no loss of generality that $c=0$. Let $(\eta_0, \eta_1,\eta_2,\eta_3)\in \Lambda^0\R^3\times\Lambda^1\R^3\times\Lambda^2\R^3 \times \Lambda^3 \R^3$ with $(\eta_0,\eta_1,\eta_2) \not= (0,0,0)$ be such that $\phi$ is given by~\eqref{e:phi-for-sect5}. The function $\phi$ is smooth and for all $(\theta,\omega) \in \free$ we have
\begin{equation} \label{e:dphi}
d_{(\theta,\omega)}\phi (\theta',\omega') = (\eta_2 + \eta_0 \omega) \wedge \theta'+ (\eta_1 + \eta_0 \theta) \wedge \omega'~.
\end{equation}
If $\eta_0 =0$, we have $(\eta_1,\eta_2) \not= (0,0)$, therefore $d_{(\theta,\omega)}\phi \not=0$ for all $(\theta,\omega) \in \free$. It follows that $S$ is a smooth 5-dimensional submanifold of $\free$ (it is actually a codimension-1 affine subspace of $\free$) and so is the relatively open subset $\Nonchar(S)$ of $S$. If $\eta_0\not=0$ then $d_{(\theta,\omega)}\phi \not=0$ for all $(\theta,\omega) \in  \free \setminus \{(- \eta_0^{-1} \eta_1,- \eta_0^{-1} \eta_2)\}$. Therefore $S\setminus \{(- \eta_0^{-1} \eta_1,- \eta_0^{-1} \eta_2)\}$ is a smooth 5-dimensional submanifold of $\free$. Note incidentally that $(- \eta_0^{-1} \eta_1,- \eta_0^{-1} \eta_2) \in S$ if and only if $\eta_3 = \eta_0^{-1} \eta_1\wedge \eta_2$. Let us now verify the inclusion $\Nonchar(S)  \subset S\setminus \{(- \eta_0^{-1} \eta_1,- \eta_0^{-1} \eta_2)\}$. Let $(\theta,\omega) \in \Nonchar(S)$. Then there is $\tau \in \Lone$ such that $\phi((\theta,\omega)\cdot (\tau,0)) = (\eta_2  +\eta_1 \wedge \theta  + \eta_0 \omega) \wedge \tau \not= 0$. Therefore $\eta_2  +\eta_1 \wedge \theta  + \eta_0 \omega \not=0$ which implies that $(\theta,\omega) \not= (- \eta_0^{-1} \eta_1,- \eta_0^{-1} \eta_2)$, as wanted. It follows that $\Nonchar(S)$ is a relatively open subset of $S\setminus \{(- \eta_0^{-1} \eta_1,- \eta_0^{-1} \eta_2)\}$ and hence is a smooth 5-dimensional submanifold of $\free$. To conclude the proof of the lemma, let $x= (\theta,\omega) \in \Nonchar(S)$. It follows from~\eqref{e:dphi} that for $\tau \in \Lone$, one has $\phi(x \cdot (\tau,0)) = \phi(x +(\tau,\theta \wedge \tau)) = d_{x}\phi(\tau,\theta\wedge\tau)$. Therefore $x\cdot (\tau,0) \in S$ if and only if $x\cdot (\tau,0) = x + (\tau,\theta\wedge \tau) \in T_x S$, i.e., $T_x S \cap \Hor_x = S \cap \Hor_x$. 
\end{proof}

\section{Classification in nonfree step-2 rank-3 Carnot algebras} \label{sect:nonfree-case}

This section is devoted to the proof of Theorem~\ref{thm:nonfree-case}. We recall that a Carnot morphism $\pi :\frakf \rightarrow \frakg$ between step-2 Carnot algebras $\frakf =\frakf_1 \oplus \frakf_2$ and $\frakg = \frakg_1 \oplus \frakg_2$ is a homomorphism of graded Lie algebras, which means that $\pi$ is a linear map such that $\pi([x,y]) = [\pi(x),\pi(y)]$ for all $x,y\in \frakf$ and $\pi(\frakf_i) \subset \frakg_i$ for $i=1,2$. Note that a Carnot morphism is both a homomorphism of graded Lie algebras and a group homomorphism. It can easily be seen that the preimage of a precisely monotone set under a Carnot morphism is precisely monotone. We give in the next lemma the rather elementary proof of this property, for the reader's convenience.

\begin{lemma} \label{lem:PM-vs-quotient}
Let $\frakf$ and $\frakg$ be step-2 Carnot algebras and $\pi :\frakf \rightarrow \frakg$ be a Carnot morphism. Let $E\subset \frakg$ be precisely monotone. Then $\pi^{-1}(E) \subset \frakf$ is precisely monotone.
\end{lemma}

\begin{proof}
Let $E\subset \frakg$ be precisely monotone and $\ell \subset \frakf$ be a horizontal line. Then $\pi(\ell)$ is either a singleton or a horizontal line in $\frakg$. If $\pi(\ell)$ is a singleton then either $\ell \cap \pi^{-1}(E)=\emptyset$ or $\ell \cap \pi^{-1}(E)=\ell$ and in both cases $\ell$ intersects both $\pi^{-1}(E)$ and its complement in a connected set. If $\pi(\ell)$ is a horizontal line in $\frakg$ then the restriction of $\pi$ to $\ell$ is a homeomorphism from $\ell$ to $\pi(\ell)$. Since $\pi(\ell)$ intersects both $E$ and $E^c$ in a connected set, it follows that $\ell$ intersects both $\pi^{-1}(E)$ and its complement in a connected set, which concludes the proof of the lemma. 
\end{proof}

To prove Theorem~\ref{thm:nonfree-case} let $\frakg$ be a step-2 rank-3 Carnot algebra that is not isomorphic to $\free$ and let $E\subset \frakg$ be precisely monotone and measurable with $E\not\in \{\emptyset,\frakg\}$. Let $\pi : \free \rightarrow \frakg$ be a surjective Carnot morphism. Recall that the universal property of the free step-2 rank-3 Carnot algebra ensures the existence of such a surjective Carnot morphism, see for instance~\cite[p.45]{MR1218884}, and $\ker \pi$ is a non trivial linear subspace of  $\Ltwo$ since $\frakg$ is not isomorphic to $\free$. By Lemma~\ref{lem:PM-vs-quotient} and since $\pi$ is continuous, $\pi^{-1}(E)$ is a precisely monotone and measurable subset of $\free$ with $\pi^{-1}(E) \not\in \{\emptyset,\free\}$. By Theorem~\ref{thm:main}, there is a non constant h-affine function $\phi:\free \rightarrow \R$ such that $\Int(\pi^{-1}(E)) = \{x\in \free :\, \phi(x) < 0\}$, $\overline{\pi^{-1}(E)} = \{x\in \free :\, \phi(x) \leq 0\}$, and $\partial \pi^{-1}(E) = \{x\in \free :\, \phi(x) = 0\}$. Since $\pi$ is linear and surjective, $\pi$ is  continuous and open. Therefore $\Int(\pi^{-1}(E)) = \pi^{-1} (\Int E)$, $\overline{\pi^{-1}(E)} = \pi^{-1}(\overline{E})$, $\partial \pi^{-1}(E) = \pi^{-1}(\partial E)$, and it follows that 
\begin{equation} \label{e:non-free-case}
\left\{
\begin{aligned}
\Int (E) &= \pi (\{x\in \free :\, \phi(x) < 0\}) \\
\overline{E} &= \pi (\{x\in \free :\, \phi(x) \leq 0\}) \\
\partial E &= \pi (\{x\in \free :\, \phi(x) = 0\}).
\end{aligned}
\right.
\end{equation}

In particular, we have $\pi (\{x\in \free :\, \phi(x) < 0\}) \cap \pi (\{x\in \free :\, \phi(x) = 0\}) = \emptyset$. We shall now verify that this implies that $\phi$ factors through $\free / \ker \pi$, i.e., $\phi(\theta+ \omega + \zeta) = \phi(\theta+\omega)$ for all $\theta\in \Lone$, $\omega \in \Ltwo$, $\zeta \in \ker \pi$. Indeed otherwise there are $\theta\in \Lone$, $\omega \in \Ltwo$, $\zeta \in \ker \pi \subset \Ltwo$ such that $\phi(\theta+ \omega + \zeta) \not= \phi(\theta+\omega)$. Then it follows from~\eqref{e:phi-introduction} that the function $t\in \R \mapsto \phi(\theta, \omega + t\zeta)$ is a degree-1 polynomial and therefore is surjective. In particular, one can find $s,t \in \R$ such that $\phi(\theta+ \omega + s\zeta)<0$ and $\phi(\theta+ \omega + t\zeta)=0$, which implies that $\pi(\theta+\omega) \in \pi (\{x\in \free :\, \phi(x) < 0\}) \cap \pi (\{x\in \free :\, \phi(x) = 0\})$ and gives a contradiction.

Since $\phi$ factors through $\free / \ker \pi$, there is $\psi:\frakg \rightarrow \R$ such that $\phi = \psi \circ \pi$ and it follows from~\eqref{e:non-free-case} that 
\begin{equation*} 
\Int (E) = \{x\in \frakg :\, \psi(x) < 0\} \subset E \subset \overline{E} = \{x\in \frakg :\, \psi(x) \leq 0\}.
\end{equation*}
Furthermore, since $\phi = \psi \circ \pi$ and $\phi$ is h-affine, then $\psi:\frakg \rightarrow \R$ is h-affine (see \cite[Lemma~2.3]{LeDonneMorbidelliRigot1}). By Theorem~\ref{thm:h-affine-maps} we get that $\psi$ is affine, and $\psi$ is non constant since $\phi$ is, which concludes the proof of Theorem~\ref{thm:nonfree-case}.

\bibliography{biblio}
\bibliographystyle{amsplain}

\end{document}